\documentclass[reqno,a4paper,11pt]{amsart}
\usepackage{hyperref}
\usepackage[top=1in, bottom=1in, left=1.2in, right=1.2in]{geometry}
\usepackage{tikz}
\usepackage[active]{srcltx}
\usepackage[curve]{xypic}
\usepackage{graphicx}
\usepackage{mathrsfs}

\usepackage{color}
\usepackage[active]{srcltx}
\usepackage{epsfig,amsmath}
\usepackage[english]{babel}
\usepackage{amsmath,amsfonts,amssymb,amscd,color,epsfig,amsthm}%eufrak
\newtheorem{newthm}{Theorem}
\newtheorem{theorem}{Theorem}[section]
\newtheorem{lemma}[theorem]{Lemma}
\newtheorem{proposition}[theorem]{Proposition}
\newtheorem{corollary}[theorem]{Corollary}

\theoremstyle{definition}
\newtheorem{definition}[theorem]{Definition}
\newtheorem{remark}[theorem]{Remark}

\numberwithin{equation}{section}

\def\CCC{{\cal C}}

\def\CCC{{\cal C}}

\def\HHH{{\cal H}}

\def\OOO{{\cal O}}

\def\UUU{{\cal U}}

\def\g{\gamma}
\def\G{\Gamma}

\def\De{\Delta}
\def\de{\delta}

\def\R{\mbox{$\mathbb R$}}

\def\C{\mbox{$\mathbb C$}}

\def\P{\mbox{$\mathbb P$}}

\def\D{\mathbb D}
\def\Q{\mathbb Q}

\def\Z{\mbox{$\mathbb Z$}}

\def\N{\mbox{$\mathbb N$}}

\def\lv{ \left(\begin{matrix} }
 \def\rv{\end{matrix}\right)}

\def\cal{\mathcal}

\def\dw{{\dw}}

\newcommand{\mylabel}[1]{\label{#1}}

\newcommand{\REFEQN}[1] { \begin{equation}\mylabel{#1} }
\newcommand{\ENDEQN}{\end{equation}}
\newcommand{\REFTHM}[1] { \begin{theorem}\mylabel{#1} }
\newcommand{\ENDTHM}{\end{theorem}}
\newcommand{\REFNTH}[1] { \begin{newthm}\mylabel{#1} }
\newcommand{\ENDNTH}{\end{newthm}}
\newcommand{\REFPROP}[1]{\begin{proposition}\mylabel{#1} }
\newcommand{\ENDPROP}{\end{proposition} }
\newcommand{\REFLEM}[1]{\begin{lemma}\mylabel{#1} }
\newcommand{\ENDLEM}{\end{lemma} }
\newcommand{\REFCOR}[1]{\begin{corollary}\mylabel{#1} }
\newcommand{\ENDCOR}{\end{corollary} }

\def\pf{postcritically-finite }

\def\ov{\overline}

\def\hf{\hat{f}}

\usepackage{tikz}
\usetikzlibrary{arrows}
\tikzstyle{every picture}=[> = to]
\tikzset{cdlabel/.style={execute at begin node=$\scriptstyle,execute at end node=$}}
\tikzset{implication/.style={double equal sign distance, -implies}}
\tikzset{biimplication/.style={double equal sign distance, implies-implies}}

\title{Perturbations of graphs for Newton maps}
\author{Yan Gao and Hongming Nie}%\footnote{The auther is supported by the grant no. 11501383 of NSFC.}}
\address{Mathemaitcal School of Sichuan University, Chengdu 610064, P. R. China}
\email{gyan@scu.edu.cn}
\address{Einstein Institute of Mathematics, The Hebrew University of Jerusalem, Israel}
 \email{hongming.nie@mail.huji.ac.il}
\date{\today}

\begin{document}
\maketitle
\begin{abstract}
We study the convergence of graphs consisting of finitely many internal rays for degenerating Newton maps. We state a sufficient condition to guarantee the convergence. As an application, we investigate the boundedness of hyperbolic components in the moduli space of quartic Newton maps. We prove that such a hyperbolic component is bounded if and only if every  element has degree $2$ on the immediate basin of each root.

\vspace{0.1cm}

\vspace{0.1cm}

\noindent{\bf AMS(2010) Subject Classification}: 37F10, 37F20.
\end{abstract}

\section{Introduction}
For $d\ge 2$, denote by ${\rm Rat}_d$ the space of rational maps of degree $d$. Via identifying coefficients, the space ${\rm Rat}_d$ is an open dense subset of the $2d+1$-dimensional complex projective space $\mathbb{P}^{2d+1}$. Hence $\P^{2d+1}$ is a natural closure of ${\rm Rat}_d$. The boundary $\partial {\rm Rat}_d:=\P^{2d+1}\setminus {\rm Rat}_d$ consists of so-called degenerate rational maps. A sequence in $\mathrm{Rat}_d$ is \emph{degenerate} if its limit is a degenerate rational map. It is of interest to understand the interplay of dynamics for a degenerate sequence and its limit.
In this paper, we explore this interplay in a significant slice of ${\rm Rat}_d$, namely Newton family. We show that under some conditions, the dynamics of $f_n$ are stable (in some sense) if $f_n$ approaches $\partial {\rm Rat}_d$ within the Newton family.

Now let us be more precise. For a degree $d\ge 2$ complex polynomial $P(z)$ with simple roots, its Newton map is defined by
$$f_P(z)=z-\frac{P(z)}{P'(z)}.$$
The map $f_P$ is a degree $d$ rational map having $d$ superattacting fixed points at the roots of $P$. We also say such points are \emph{roots} of $f_P$. Denote by $\mathrm{NM}_d$ the space of degree $d$ Newton maps. It follows that $\mathrm{NM}_d$ is a $d$-dimensional subspace in $\mathrm{Rat}_d$ and hence in $\mathbb{P}^{2d+1}$. Let $\overline{\mathrm{NM}}_d$ be the closure of $\mathrm{NM}_d$ in $\mathbb{P}^{2d+1}$. Follow DeMarco \cite{DeMarco05}, for $f\in \overline{\mathrm{NM}}_d$, in homogeneous coordinates we can write $f=H_f\hat f$, where $H_f$ is a homogeneous polynomial and $\hat f$ is a rational map of degree at most $d$. We are interested in the case that $\hat f$ has degree at least $2$. Then in our case $\hat f$ is a Newton map for a polynomial with possible multiple roots and $H_f$ records the multiplicities of the fixed points of $\hat f$. For more details, we refer \cite{Nie-thesis}.

Let $f=H_f\hat f\in\overline{\mathrm{NM}}_d$ with $\deg \hat f\ge 2$. Consider the roots of $\hat f$ and the corresponding basins. Let $\mathcal{U}$ be a forward invariant set consisting of finite components of such basins, that is $\hat f(U)\in\mathcal{U}$ for $U\in\mathcal{U}$. Assume that $\hat f$ is postcritically finite in $\mathcal{U}$. Then for each $U\in\mathcal{U}$, the inverse $\psi:\mathbb{D}\to\ U$ of a B\"ottcher coordinate defines the center $u=\psi(0)$ and internal rays $I_{(U,u)}(t)$ of $\hat f$ in $U$ for $t\in\mathbb{R}/\mathbb{Z}$. Since the boundary of $U$ is locally connected \cite{Drach18, Wang18}, the internal rays in $U$ land on $\partial U$. Let $\Gamma$ be a connected graph consisting of finitely many preperiodic internal rays in elements of $\UUU$. The canonical paradigm of such graphs are the Newton graphs (see Section \ref{sub:graph}) formulated recently by Drach e.t \cite{Drach19}
 and alternative graphs for cubic Newton maps (see Section \ref{sec:cubic}) based on Roesch's work in \cite{Roesch08}.

Since $f\in \overline{\mathrm{NM}}_d$, let $\{f_n\}_{n\ge 1}\subset{\mathrm{NM}}_d$ be a sequence such that $f_n$ converges to $f$. If the convergence is under the dynamically weak Carath\'eodory topology, see Definition \ref{def:dyn-conv}, a B\"ottcher coordinate of $\hat f$ on $U\in\mathcal{U}$ naturally deduces a B\"ottcher coordinate of $f_n$ on the deformation $U_n$ of $U$, see Section \ref{sec:Bottcher}. Then we can define the corresponding internal rays in $U_n$, which either land on $\partial U_n$ or terminate at an iterated preimage of a critical point in $U_n$, see Section \ref{sec:rays}. It follows that we obtain a perturbation $\Gamma_n$ of $\Gamma$. For examples satisfying the above conditions, see Lemma \ref{lem:criterion}.

Our main result states that under natural conditions the graphs $\Gamma_n$ converge to $\Gamma$ in the Hausdorff metric topology.

\begin{theorem}\label{main}
Let $f$ and $\mathcal{U}$ be as above. Assume $f_n\in\mathrm{NM}_d$ such that $f_n$ converges to $f$, as $n\to\infty$, under the dynamically weak Carath\'eodory topology in $\UUU$. Let $\mathcal{V}\subseteq\mathcal{U}$ and $T\subseteq \mathbb{Q}$ be finite subsets and suppose that $\Gamma:=\bigcup_{(U,t)\in\mathcal{V}\times T}I_{(U,u)}(t)$ is a connected graph. 
Denote by $(U_n,u_n)$ the deformations of $(U,u)$ and assume each induced internal ray $I_{(U_n,u_n)}(t)$ lands on $\partial U_n$ for large $n$. If the orbits of the landing points of rays $I_{(U,u)}(t)$s are eventually repelling periodic and avoid the critical points of $\hat{f}$, then, for all large $n$, the graph $$\Gamma_n=\bigcup_{(U,t)\in\mathcal{V}\times T}I_{(U_n,u_n)}(t)$$ is homeomorphic to $\Gamma$ and $\Gamma_n\to\Gamma$ as $n\to\infty$.
\end{theorem}

The technic of perturbations of internal rays is widely used in complex dynamics for the non-degenerate maps, see e.g. \cite{Gao17,Goldberg93,Roesch00}. Theorem \ref{main} generalizes this technic to the degenerate case within the Newton family. The key point of the proof, differing from the non-degenerate case, is an elaborate argument to the internal rays landing at holes of $f$.

Naturally, our theorem provides a way to study degenerate sequences of Newton maps in the parameter space and hence that in moduli space.  Roughly speaking, Theorem \ref{main} asserts that, under some conditions, part of the dynamics of the degenerate map $\hat f$ embeds into the dynamics of non-degenerate maps $f_n$s. Then it allows us to control the dynamics of $f_n$ by that of $\hat f$.

As an application, we study the boundedness of hyperbolic components in the muduli space of quartic Newton maps. Recall that a rational map is \emph{hyperbolic} if each critical point converges under iteration to a (super)attracting cycle, equivalently, it is uniformly expanding in a neighborhood of its Julia set, see \cite[Section 3.4]{McMullen94}. The space of hyperbolic maps is open and conjecturally dense in the space of rational maps. Each component is a \emph{hyperbolic component}. Moreover,  the space of hyperbolic maps descends an open subset in the moduli space of rational maps, and each component of this subset is a \emph{hyperbolic component} in moduli space.

Since the point $\infty$ is the unique repelling fixed point of Newton maps, the moduli space of degree $d$ Newton maps is defined by
$$\mathrm{nm}_d:=\mathrm{NM}_d/\mathrm{Aut}(\mathbb{C}),$$
modulo the conjugacy of affine maps. We say a hyperbolic component $\mathcal{H}\subset\mathrm{nm}_d$ is of \emph{immediate escaping type} if each element in $\mathcal{H}$ has degree at least $3$ in the immediate basin of some root. For a complete classification of hyperbolic components in $\mathrm{nm}_4$, we refer \cite{Nie18} or see Section \ref{classification}.

A hyperbolic component $\mathcal{H}\subset\mathrm{nm}_d$ is \emph{bounded} if $\mathcal{H}$ has compact closure in $\mathrm{nm}_d$. Otherwise, we say $\mathcal{H}$ is \emph{unbounded}. Our next result gives all the bounded hyperbolic components in $\mathrm{nm}_4$.

\begin{theorem}\label{thm:bdd hyp}
Let $\mathcal{H}\subset\mathrm{nm}_4$ be a hyperbolic component. Then $\mathcal{H}$ is unbounded if and only if $\mathcal{H}$ is of immediate escaping type.
\end{theorem}

Bounded hyperbolic components in a more general setting appear in the literature. For a hyperbolic component in the moduli space of bicritical rational maps, if each element possesses two distinct attracting cycles, each of period at least $2$, then it is bounded, see \cite[Theorem 1]{Epstein00} and \cite[Theorem 1.1]{Nie19}. In $\mathrm{nm}_4$, the second author and Pilgrim proved that a hyperbolic component is bounded if each element has two distinct attracting cycles \cite[Main Theorem]{Nie18}. All the previous bounded results are about the case what so-called of type D, that is each element has maximal number of (super)attracting cycles. Our above result gives the first non type D bounded hyperbolic components of complex dimensions at least $2$ and strengthens the result \cite[Theorem 1.3]{Nie18}.
\begin{figure}[h]
  \includegraphics[width=.7\linewidth]{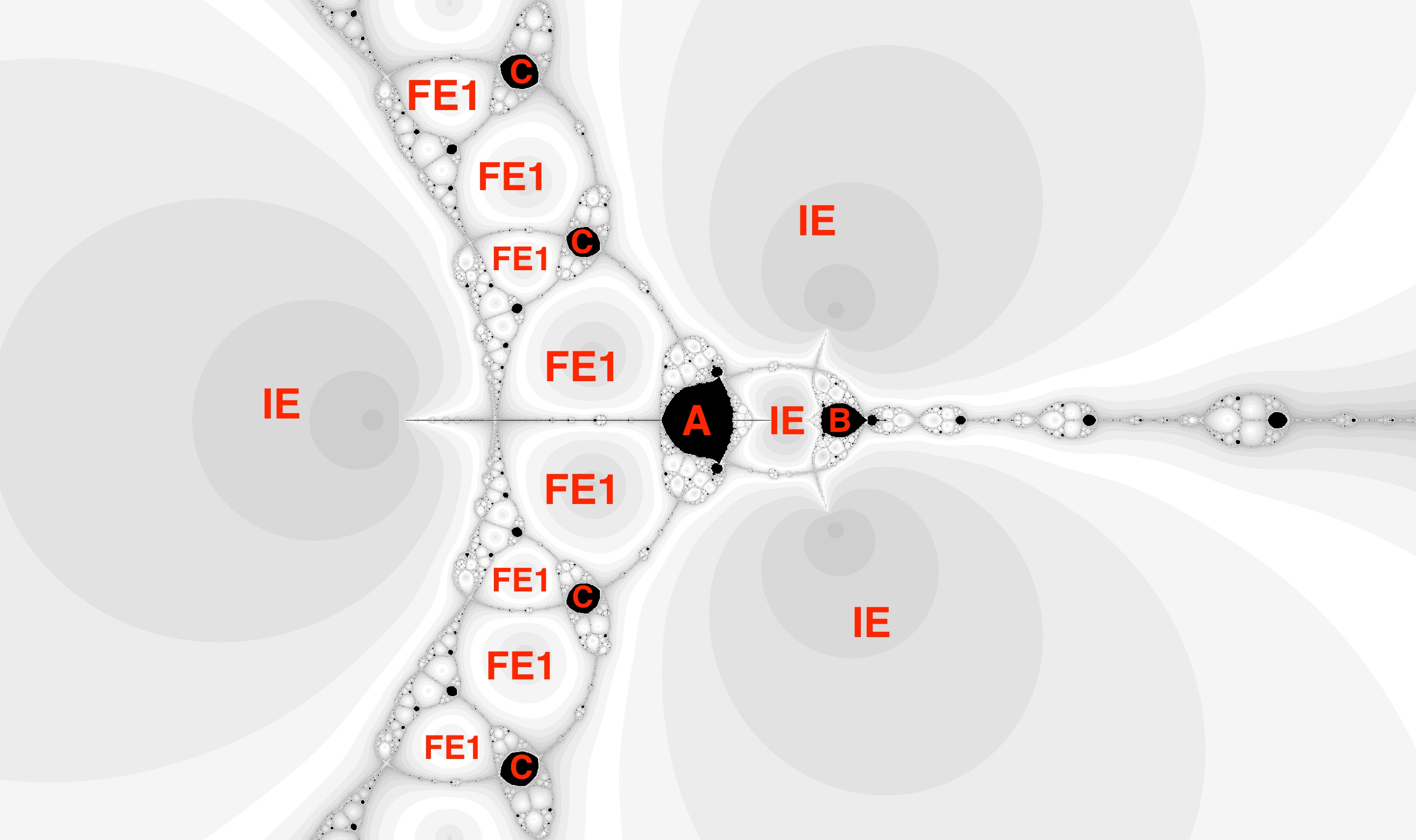}
  \caption{The locus $\mathrm{Per_2(0)}\cap\mathrm{nm}_4$, showing part of $c$-plane for the family of Newton maps $f_{P_c}$ for the polynomial $P_c(z)=z^4/12-cz^3/6+(4c-3)z/12+(3-4c)/12$, see \cite[Figure 1]{Nie18}. The critical points of $f_{P_c}$ are the four roots of $P_c(z)$, $0$ and $c$. The periodic critical orbit is $0\to 1\to 0$. The letters indicate the types of hyperbolic components, see Section \ref{classification}. Our result asserts that the hyperbolic components indicated by  A, B, C, or FE1 are bounded in $\mathrm{nm}_4$}
  \label{per_2}
\end{figure}

One direction of Theorem \ref{thm:bdd hyp} is the result \cite[Theroem 1.4 ]{Nie18}: if $\mathcal{H}$ is of immediate escaping type, then $\mathcal{H}$ is unbounded. In this paper, we prove the reverse implication. Differing from the analytic argument in \cite{Epstein00} and the arithmetic argument in \cite{Nie18} and \cite{Nie19}, our argument relies on the combinatorial properties of Newton maps and applies Theorem \ref{main}. The proof goes by contradiction as follows. Suppose $\mathcal{H}$ is unbounded and not of immediate escaping type. Then we obtain a unbounded sequence $[f_n]\in\mathcal{H}$. Moreover, passing to subsequences, $[f_n]$ has a lift $f_n\in\mathrm{NM}_4$ such that $f_n$ converges to $f=H_f\hat f$ with $\deg\hat f=2$ or $3$ and no roots of $f_n$ colliding as $n\to\infty$, see Lemma \ref{lem:lift}. It follows that at least one critical point $c_n$ of $f_n$ diverging to $\infty$. If $\deg\hat f=2$, consider rational internal rays in the immediate basins of the roots of $\hat f$ and the corresponding perturbations for $f_n$. Theorem \ref{main} implies that $\deg f_n=2$ and hence leads to a contradiction. If $\deg\hat f=3$ and $\mathcal{H}$ is of type A, B, C or D, it turns out that the Newton graphs of $\hat f$ are disjoint with the unique non-fixed critical point $c$. Applying Theorem \ref{main} to the Newton graphs of $\hat f$, we bound the immediate basins of the (super)attracting cycles of periods at least 2 for $f_n$. We obtain a contradiction by considering the location of forward orbit of the critical point $c_n$. In the remaining case that $\deg\hat f=3$ and $\mathcal{H}$ is of type FE1 or FE2, in general the critical point $c$ may be an iterated preimage of $\infty$. Then we can not apply Theorem \ref{main} to the perturbations of Newton graphs as in previous case. Our strategy in this case is as follows. Applying Rosech's result (see \cite{Roesch08}) on cut angles, we construct a natural Jordan curve $\CCC$ consisting of (pre)periodic internal rays of $\hat f$ such that the orbit of $\CCC$ is away from the critical point $c$. We perturb such curve $\mathcal{C}$ of $\hat f$ and obtain curves $\mathcal{C}_n$ for $f_n$. Then by Theorem \ref{main}, we have $\mathcal{C}_n$ converges to $\mathcal{C}$. By analyzing the locations of the related critical points and the corresponding Fatou components of $f_n$, for any fixed map in this sequence, we can lift a natural arc under the iteration of this map. To obtain a contradiction, we show such lifts have a positive length in the limit by the above locations argument.

Our proof of Theorem \ref{thm:bdd hyp} highly relies on the control of the orbits of critical points, see Lemma \ref{lem:key}. We do not expect that an analogy of such control works for Newton maps of higher degrees. But it would still be interesting to use Theorem \ref{main} to investigate the boundedness of hyperbolic components in $\mathrm{nm}_d$ for $d\ge 5$.

In principle, our main result (Theorem \ref{main}) is supposed to be efficient to deal with the degenerating Newton maps having only roots diverging to $\infty$ and no roots colliding. Note that any sequence in the moduli space $\mathrm{nm}_d$ can be lifted to such a sequence (Lemma \ref{lem:lift}), Theorem \ref{main} thus provides a useful tool to study the boundary behavior of $\mathrm{nm}_d$. It would be also interesting to develop an analogy of Theorem \ref{main} concerning the collision of roots.

This paper is organized as follows. In Section \ref{pre}, we introduce the relevant preliminaries about degenerate rational maps and Newton maps. In particular, in Section \ref{sec:cubic}, we state Roesch's result on cut angles for cubic Newton maps and construct related graphs. In Section \ref{sec:quartic}, we generalize the cut angles to quartic Newton maps. Section \ref{sec:pf main} contains the proof of Theorem \ref{main}. In Section \ref{sec:bdd}, we prove Theorem \ref{thm:bdd hyp} by a case-to-case argument.

\subsection*{Acknowledgements}
We thank Kevin Pilgrim for fruitful discussion and useful comments on an early draft. This work was discussed when both authors visited the Indiana University Bloomington (IUB).
We are grateful to the Department of Mathematics at IUB for its hospitality. The first author is partially supported by NSFC grant no. 11871354.

\section{Preliminaries}\label{pre}
In this section, we give background materials. In Section \ref{deg-rat}, we provide basic definitions and properties of degenerate rational maps. Section \ref {sub:Newton} contains the properties of Newton maps. Section \ref{sub:graph} introduces the \emph{Newton graphs} given by Drach. et \cite{Drach19}. In Section \ref{sec:cubic}, we first state Roesch's result on cut angles and then construct invariant graphs differing from the Newton graphs for cubic Newton maps. In Section \ref{sec:quartic}, we generalize Roesch cut angles result to quartic Newton maps.

\subsection{Degenerate rational maps}\label{deg-rat}
As mentioned in the introduction, the space $\mathrm{Rat}_d$ is naturally identified to an open dense subset of $\mathbb{P}^{2d+1}$. We say each element $f\in\mathbb{P}^{2d+1}\setminus\mathrm{Rat}_d$ is a \emph{degenerate rational map} of degree $d$. For such $f$, there exist two degree $d$ homogeneous polynomials $F(X,Y)$ and $G(X,Y)$ in $\mathbb{C}[X,Y]$ such that $f=[F:G]$ in homogeneous coordinate. We can rewrite
$$f=H_f\hat f,$$
where $H_f=\gcd[F,G]$ and $\hat f$ is a rational map of degree less than $d$. We say each zero of $H_f$ is a \emph{hole} of $f$ and denote by $\mathrm{Hole}(f)$ the set of holes of $f$. Moreover, we call $\hat f$ the \emph{reduction} of $f$.  For connivence, if $f$ is a rational map of degree $d$, we define $H_f=1$ and then $\hat f=f$.

Let $\{f_n\}_{n\ge 1}$ be a sequence of rational maps of degree $d\ge 1$. We say $f_n$ converges \emph{semi-algebraically} to a (degenerate) rational map $f$ if the coefficients of $f_n$ converge to the coefficients of $f$ in $\mathbb{P}^{2d+1}$. The semi-algebraical convergence implies locally uniform convergence away from holes:

\begin{lemma}\cite[Lemma 4.1]{DeMarco05}\label{lem:semi-convergence}
Let $\{f_n\}_{n\geq 1}$ be a sequence of degree $d\ge 1$ rational maps. If $f_n$ converges semi-algebraically to $f=H_f\hat{f}\in\P^{2d+1}$. Then $f_n$ converges locally uniformly to $\hat{f}$ outside  $\mathrm{Hole}(f)$.
\end{lemma}

Conversely, combining \cite[Lemma 2.8]{Cui18} and \cite[Lemma 2.2.3]{Nie-thesis}, we have
\begin{lemma}\label{lem:convergence1}
Let $\{f_n\}_{n\geq 1}$ be a sequence of degree $d\ge 1$ rational maps and let $S\subset\mathbb{P}^1$ be a finite set. Suppose $f_n$ converges locally uniformly to a map $\hat{f}$ on $\mathbb{P}^1\setminus S$. Then $\hat f$ is rational map of degree at most $d$. Moreover, there exists a homogeneous polynomial $H_f$ of degree $d-\deg\hat f$ such that $f_n$ converges semi-algebraically to $f:=H_f\hat f$ and $\mathrm{Hole}(f)\subset S$.
\end{lemma}

Suppose each $f_n$ possesses a cycle of fixed period. If the limit of these cycles is away from the holes of $f$, Lemma \ref{lem:semi-convergence} immediately implies that this limit is also a cycle for $\hat f$. We state as follows and omit the proof.
\begin{lemma}\label{lem:limit cycle}
Let $\{f_n\}_{n\geq1}$ be a sequence of degree $d\ge 2$ rational maps. Suppose that $f_n$ converges semi-algebraically to $f=H_f\hat{f}\in\P^{2d+1}$ with $\deg\hat f\ge 1$. Assume $\mathcal{O}_n$ is a cycle of $f_n$ of period $m\ge 1$ and suppose that $\mathcal{O}_n$ converges to $\mathcal{O}$ in $\mathbb{P}^1$. If $\mathcal{O}\cap\mathrm{Hole}(f)=\emptyset$, then $\mathcal{O}$ is a cycle of $\hat f$ of period $q$ with $q\mid m$. Furthermore, (1) if $\mathcal{O}_n$ is attracting, then $\mathcal{O}$ is non-repelling; (2) if $q<m$, then $\mathcal{O}$ is parabolic.
\end{lemma}
If the limit intersects the holes of $f$, we have the following basins shrinking result.

\begin{lemma}\cite[Proposition 2.8]{Nie18}\label{lem:2.8}
Let $\{f_n\}_{n\geq1}$ be a sequence of degree $d\ge 2$ rational maps. Assume that $f_n$ converges semi-algebraically to $f=H_f\hat{f}\in\P^{2d+1}$. Assume ${\deg}(\hat{f})\geq2$ and $\infty\in\mathrm{Hole}(f)$ is a fixed point of $\hat{f}$. Let $\{z_n^{(0)},\dots,z_n^{(m-1)}\}$ be a (supper)attracting cycle of $f_n$ of period $m\geq2$, and let $U_n^{(k)}$ be the Fatou component containing $z_n^{(k)}$. Suppose $z_n^{(k)}\to z^{(k)}$ for $k=0,\ldots,m-1$ with $z^{(0)}=\infty$ and $z^{(i)}\not=\infty$ for some $1\leq i\leq m-1$. Then
\begin{enumerate}
\item $U_n^{(0)}$ converge to $\infty$ in the sense that, for any $\epsilon>0$, the component $U_n^{(0)}$ is contained in the disk $\{z:\rho(z,\infty)<\epsilon\}$ for all large $n$, where $\rho$ is the sphere metric; and
\item there exists a neighborhood $V$ of $\infty$ such that $U_n^{(i)}\cap V=\emptyset$ for all large $n$.
\end{enumerate}
\end{lemma}

Now we state a straightforward result about the perturbations of periodic points.

\begin{lemma}\label{lem:periodic-point}
Let $f=H_f\hat f\in\P^{2d+1}$ with $\deg\hat f\ge 1$. Then the following holds.
\begin{enumerate}
\item  For $z_0\in\widehat{\C}$ and $j\ge 1$, denote by $z_i:=\hat f^i(z_0)$ for $0\le i\le l$. 
Suppose $z_i$ avoids the critical point of $\hat f$ for all $0\le i\le j-1$. Let $z_j(g)$ be a holomorphic map defined in a neighborhood of $f\in\P^{2d+1}$ with $z_j(f)=z_j$. Then for each $0\le i\le\ldots,j-1$, there exists a holomorphic map $z_i(g)$ defined in a neighborhood of $f$ such that $z_i(f)=z_i$ and $\hat g^{j-i}(z_i(g))=z_j(g)$. Moreover, if $z_i$ avoids the holes of $f$ for all $0\le i\le\ldots,j-1$, then $z_i(g)$ is the unique point near $z_i$ satisfying $\hat g^{j-i}(z_i(g))=z_j(g)$, which implies  $z_i(g)=\hat g^i(z_0(g))$ for $i=0,\ldots,j-1$.
\item Let $\mathcal{O}=\{\xi_0,\dots,\xi_{k-1}\}$  be an attracting (resp. repelling) cycle of $\hat f$. If $\mathcal{O}\cap\mathrm{Hole}(f)=\emptyset$, then for each $g$ close to $f$, there exists a unique attracting (resp. repelling) cycle $\mathcal{O}(g):=\{\xi_0(g),\dots,\xi_{k-1}(g)\}$ of $g$ such that each $\xi_i(g)$ is a holomorphic map near $f$ with $\xi_i(f)=\xi_i$.
\end{enumerate}
\end{lemma}

\begin{proof}
By pre and post composition of  M\"{o}bius transformations, we can assume $z_0,\ldots,z_j\in\C$. For $g=H_g\hat g\in\mathbb{P}^{2d+1}$ close to $f$, we have $\deg\hat g\ge 1$. Then for $0\le i\le j-1$, the iteration $g^{j-i}$ is well-defined, see \cite[Lemma 2.2]{DeMarco05}. Consider the holomorphic function $F_i(g,z):=g^{j-i}(z)-z_j(g)$ on $\Lambda_f\times D(z_j)$ where $\Lambda_f\subseteq \P^{2d+1}$ is a neighborhood of $f$ and $D(z_j)\subseteq\C$ is a neighborhood of $z_j$. By the assumptions, we have that $F_i(f,z_i)=0$ and
\[\frac{\partial F_i}{\partial z}|_{(f, z_i)}=(\hat f^{j-i})'(z_i)\not=0.\]
Then the Implicit Function Theorem implies there exists a holomorphic function $z_i(g)$ near $f$ satisfying $\hat g^{j-i}(z_i(g))=z_j(g)$. If $\{z_0,\ldots,z_{j-1}\}\cap\mathrm{Hole}(f)=\emptyset$, the function $\hat g^{j-i}(z)$ is holomorphic in $z$ in a fixed neighborhood of $z_i$ for each $g$ close to $f$. It follows from Hurwitz's Theorem (see \cite{Gamelin01}) that  $g^{j-i}(z)-z_j(g)$ has a unique root near $z_i$ for $g$ close to $f$. Thus statement (1) follows.

For statement (2), note that the cycle $\mathcal{O}\cap\mathrm{Hole}(f)=\emptyset$. Applying the Implicit Function Theorem on $G(g,z):= g^k(z)-z$, we obtain the expected cycle $\mathcal{O}(g)$ of  $ g$ for $g$ close to $f$.
\end{proof}

For $f=H_f\hat f\in\mathbb{P}^{2d+1}$, assume $\hat f$ has an attracting cycle $\mathcal{O}$ and denote by $\Omega$ the immediate basin of $\mathcal{O}$. If $\Omega\cap\mathrm{Hole}(f)=\emptyset$, Lemma \ref{lem:periodic-point} implies that for $g$ close to $f$, the map $\hat g$ has an attracting cycle $\mathcal{O}(g)$. Denote by $\Omega(g)$ the immediate basin of $\mathcal{O}(g)$. The we have

\begin{lemma}\label{lem:fatou}
Let $E\subset\Omega$ be any compact set. Then $E\subseteq\Omega(g)$ for any $g$ sufficiently close to $f$.
\end{lemma}
This above result is well-known in the case that $f$ is a rational map of degree $d$, see \cite[Lemma 6.3]{Douady94}. Our assumption $\Omega\cap\mathrm{Hole}(f)=\emptyset$ guarantees that the argument in the non-degenerate case also works in our case. Here, we omit the proof.

\subsection{Newton maps}\label{sub:Newton}
For a degree $d\ge 2$ complex polynomial $P(z)$ with simple roots, its Newton map
$$f_P(z):=z-\frac{P(z)}{P'(z)}$$
is a degree $d$ rational map having $d$ superattracting fixed points at the roots of $P$. The only other fixed point is at $\infty$. The Holomorphic Index Formula (see \cite[Theorem 12.4]{Milnor06B}) asserts that the point $\infty$ is the unique repelling fixed point of $f_P$. The critical points of $f_P$ are the roots of $P$ and the zeros of $P''$. Moreover, the poles of $f_p$ are the zeros of $P'$.

Recall that $\mathrm{NM}_d$ is the space of degree $d$ Newton maps and $\overline{\mathrm{NM}}_d$ is the closure of $\mathrm{NM}_d$ in $\mathbb{P}^{2d+1}$. Then for each $f=H_f\hat f\in\overline{\mathrm{NM}}_d$, there exists a degree at most $d$ polynomial $Q$ with possible multiple roots such that $\hat f$ is the Newton map of $Q$. Each root $r$ of $Q$ is a (super)attracting fixed point of $\hat f$ with multiplier $1-1/n_r$, where $n_r$ is the multiplicity of $r$ as a zero of $Q$. Moreover, again $\hat f$ has a unique repelling fixed point at $\infty$ and has no other fixed points. It follows that each hole of $f$ is either a multiple root of $Q$ or $\infty$. Furthermore, $\infty\in\mathrm{Hole}(f)$ if and only if $\deg Q<d$. For more details about degenerate Newton maps, we refer \cite{Nie-thesis}.

Conversely, the following result, which is originally due to Head \cite{Head87}, gives a criterion to determine whether a rational map is a reduction of a (degenerate) Newton map. The criterion concerns only the fixed points and the corresponding multipliers.

\begin{proposition}
A rational map $\hat g$ of degree $d\geq2$ is a reduction of a (degenerate) Newton map of degree at least $d$ if and only if $\hat g$ has $d+1$ distinct fixed points $r_1,\ldots, r_d,\infty$ such that each $r_i$ has multiplier of the form $1-1/n_i$ with $n_i\in\N$.
\end{proposition}

For $f=H_f\hat f\in\overline{\mathrm{NM}}_d$ with $\deg\hat f\ge 2$, the Fatou components of $\hat f$ have well-studied topological structure. According to Shishikura \cite{Shishikura09}, all Fatou components of $\hat f$ are simply connected, and hence the Julia set of $\hat f$ is connected.  Moreover, the boundary of each component of the basins of roots of $\hat f$ is locally connected, see \cite{Drach18} and \cite{Wang18}.

\subsection{Newton graphs}\label{sub:graph}
Let $f\in\mathrm{NM}_d$ with $d\ge 2$. Denote by $\Omega_{ f}$ the union of basins of the roots of $ f$, i.e., $z\in\Omega_f$ if the orbit of $z$ converges to a root of $f$. We say $ f$ is \emph{postcritically finite in $\Omega_{ f}$} if each critical point of $ f$ in $\Omega_{ f}$ has finite orbit. The dynamics of $ f$ can be characterized by an invariant graph what is so-called Newton graph. Such graph was first constructed in \cite{Drach19} and then applied to study the dynamics of corresponding maps, see \cite{Drach18, Gao19b, Gao19a, Lodge15b, Lodge15a, Wang18}. In this subsection, we state briefly the construction of Newton graphs and list some properties.

Since $ f$ is postcritically finite in $\Omega_{ f}$, the B\"ottcher coordinates and hence the internal rays give a natural dynamical descriptions for each component of $\Omega_{ f}$. For details, we refer \cite{Milnor06B}. Let $r$ be a root of $ f$ and denote $\Omega_{ f}(r)$ its immediate attracting basin.
The fixed internal rays in $\Omega_{ f}(r)$ land at fixed points in $\partial \Omega_{ f}(r)$. Since the only Julia fixed point of $ f$ is at $\infty$, all fixed internal rays in $\Omega_{ f}$ have a common landing point at $\infty$. We denote $\De_0$ the union of all fixed internal rays in $\Omega_{ f}$ together with $\infty$. Then $ f(\De_0)=\De_0$. For any $m\geq0$, denote by $\De_m$ the connected component of $ f^{-m}(\De_0)$ that contains $\infty$. Following \cite{Drach19}, we call $\De_m$  the \emph{Newton graph} of $ f$ at level $m$. The vertex set $V_{\De_n}$ of $\De_n$ consists of iterated preimages of fixed points of $ f$ contained in $\De_n$.

A crucial property for Newton graphs is the following.
\begin{lemma}\cite[Theorem 3.4]{Drach19}\label{lem:connected}
There exists $M\geq0$ such that the Newton graph  $\De_M$ contains all poles of $ f$. Hence $\De_{m+1}=f^{-1}(\De_m)$ and $\De_m\subseteq\De_{m+1}$ for any $m\geq M$.
\end{lemma}

The Newton graphs induce naturally a puzzle structure for $ f$ on $\widehat{\C}$. Let $\De_{ f}$ denote the Newton graph of $ f$ with the least level such that $\De_{ f}$ contains all poles and all critical points that map to fixed points under iteration. Set $X_0$ the complement of the union of the disks $\{z\in U:\phi_U(z)<1/2\}$ for all connected components $U$ of $\Omega_{ f}$ with $U\cap \De_f\not=\emptyset$, where $\phi_U$ is the B\"ottcher coordinate on $U$. Define $G_0:=(\De_f\cap X_0)\cup \partial X_0$. Then $G_0$ is a finite graph consisting of segments of internal rays and equipotential lines in $\Omega_{ f}$. For each $m\geq 0$, we define $X_m:={f}^{-m}(X_0)$ and $G_m:={f}^{-m}(G_0)$. Then each $X_m$ is connected and the interior ${\rm int}(X_m)$ contains the Julia set $J_{ f}$ of $ f$. For each $m\geq0$, the closures of the components of $X_m\setminus G_m$ are called \emph{puzzle pieces of level m}. It follows that the puzzle pieces of different levels have a nested structure. For each $z\in J_{ f}$, denote $E_m(z)$ the union of puzzle pieces of level $m$ which contains $z$. Then $z\in{\rm int}(E_m(z))$. Moreover,  $E_m(z)$ are puzzle pieces for all $m$ if and only if $z$ is not an iterated preimage of $\infty$.

\begin{proposition}\cite{Drach18, Wang18}\label{thm:WYZ}
If $z$ is on the boundary of a component of $\Omega_{{f}}$, then $\cap_{m\geq0}E_m(z)=\{z\}$. In particular, the boundary of any component of basins of the roots is locally connected.
\end{proposition}

\subsection{An alternative graph for cubic Newton maps}\label{sec:cubic}
In this subsection, we  focus on the case that $f\in\mathrm{NM}_3$ is a cubic Newton map. Except some special cases, we construct an invariant graph away from the unique non-fixed critical point. Our graph is based on Roesch's work in \cite[Section 3]{Roesch08} and differs from the Newton graphs.

Let $r_1,r_2$ and $r_3$ be the roots of $f$ and let $\Omega_1$, $\Omega_2$ and $\Omega_3$ be the corresponding immediate basins, respectively.
Note that $f$ has another critical point denoted by $c$. In this subsection, without emphasis, we always assume the $c\not\in\Omega_1\cup\Omega_2\cup\Omega_3$ and $c$ is not a pole, that is $f(c)\not=\infty$.

Under the assumptions, we have that $f$ has two distinct poles, denoted by $\xi_1$ and $\xi_2$. An orientation argument implies that $\Omega_1$, $\Omega_2$ and $\Omega_3$ can not intersect at a common pole. It follows that both $\xi_1$ and $\xi_2$ are contained in $\partial\Omega_1\cup\partial\Omega_2\cup\partial\Omega_3$. By counting the preimages of $\Omega_i$s, we have that there is a unique pole at which exact two of $\partial\Omega_i$s intersect. We set $\xi_1$ to be this pole and let $\Omega_1$ and $\Omega_2$ be the immediate basins whose boundaries contain $\xi_1$. It follows that $\xi_2\in\partial\Omega_3$ and $\xi_2\not\in\partial\Omega_1\cup\partial\Omega_2$.

For $i=1,2$ and $3$, denote $I_i(\theta)$ the internal ray in $\Omega_i$ of angle $\theta\in\mathbb{R}/\mathbb{Z}$. Following Roesch \cite{Roesch08}, we say an angle $\theta$ is a \emph{cut angle} in $\Omega_1$ if there exists $\theta'\in\mathbb{R}/\mathbb{Z}$ such that $I_1(\theta)$ and $I_2(\theta')$ land at a common point. It turns out that $\theta$ is a cut angle in $\Omega_1$ if and only if $1-\theta$ is a cut angle in $\Omega_2$. For the basin $\Omega_3$, the only cut angle is $0$. Let $\Theta$ be the set of cut angles in $\Omega_1$. It follows immediate that $0,1/2\in\Theta$. Define
$$\alpha:=\inf\{\theta:\theta\in\Theta\},$$
where $\inf$ is obtained under the order by identifying $\mathbb{R}/\mathbb{Z}$ with $(0,1]$. In fact, the locally connectivity of $\partial\Omega_1$ and $\partial\Omega_2$ implies that $\alpha\in\Theta$.

Now we summarize the properties of the cut angles for later use. We use the following notations. Let $\Omega_i^{(1)}$ be the preimage of $\Omega_i$ disjoint from $\Omega_i$. Then $c\not\in\Omega_i^{(1)}$. For $j\ge 1$, if $\Omega_i^{(j)}$ is a domain such that $f^j:\Omega_i^{(j)}\to \Omega_i$ is a homeomorphism, then an internal ray $I_i(\theta)$ in $\Omega_i$ deduces an internal ray $I_i^{(j)}(\theta)$ in $\Omega_i^{(j)}$ satisfying $I_i^{(j)}(\theta)=f^{-j}(I_i(\theta))$.

\begin{lemma}\cite[Section 3]{Roesch08}\label{lem:cubic}
Fix the notations as above. The following hold.
\begin{enumerate}
\item If the orbit of a rational angle $\theta$ is contained in $[\alpha,1]$, then $\theta\in \Theta$.
\item The angle $0<\alpha<1/2$. Furthermore, the periodic angles $1-\frac{1}{2^n-1}$ belong to $\Theta$ for all large $n$.
\item Assume $0<\theta<1/2$ with $2\theta\in \Theta$. Then $\theta+1/2\in \Theta$. Furthermore, if $\theta\in \Theta$, then $I_1^{(1)}(2\theta)$ and $I_2^{(1)}(1-2\theta)$ land at a common point; if $\theta\not\in\Theta$, then $I_1(\theta)$ and $I_{2}^{(1)}(1-2\theta)$ land at a common point, as well as $I_2(1-\theta)$ and $I_1^{(1)}(2\theta)$.  The two landing points are distinct.
\item The curve $$\g(0,1/2):=I_1(0)\cup I_1(1/2)\cup I_2(0)\cup I_2(1/2)$$ separates $\Omega_3,$ and $\Omega_3^{(1)}$.
\item Let $0<\theta<1/2$ with $2\theta\in\Theta$. If $\theta\not\in \Theta$, then the curve
$$I_1(1/2)\cup I_1(\theta)\cup I_2^{(1)}(1-2\theta)\cup I_2^{(1)}(0)\cup I_1^{(1)}(0)\cup I_1^{(1)}(2\theta)\cup I_2(1-\theta)\cup I_2(1/2)$$
separates $c$ and $\infty$.
\item If $f(c)=\infty$, then the periodic angles in $\Theta$ are just $0$.
\end{enumerate}
\end{lemma}

\begin{figure}[h]
  \centering
 \includegraphics[width=.5\linewidth]{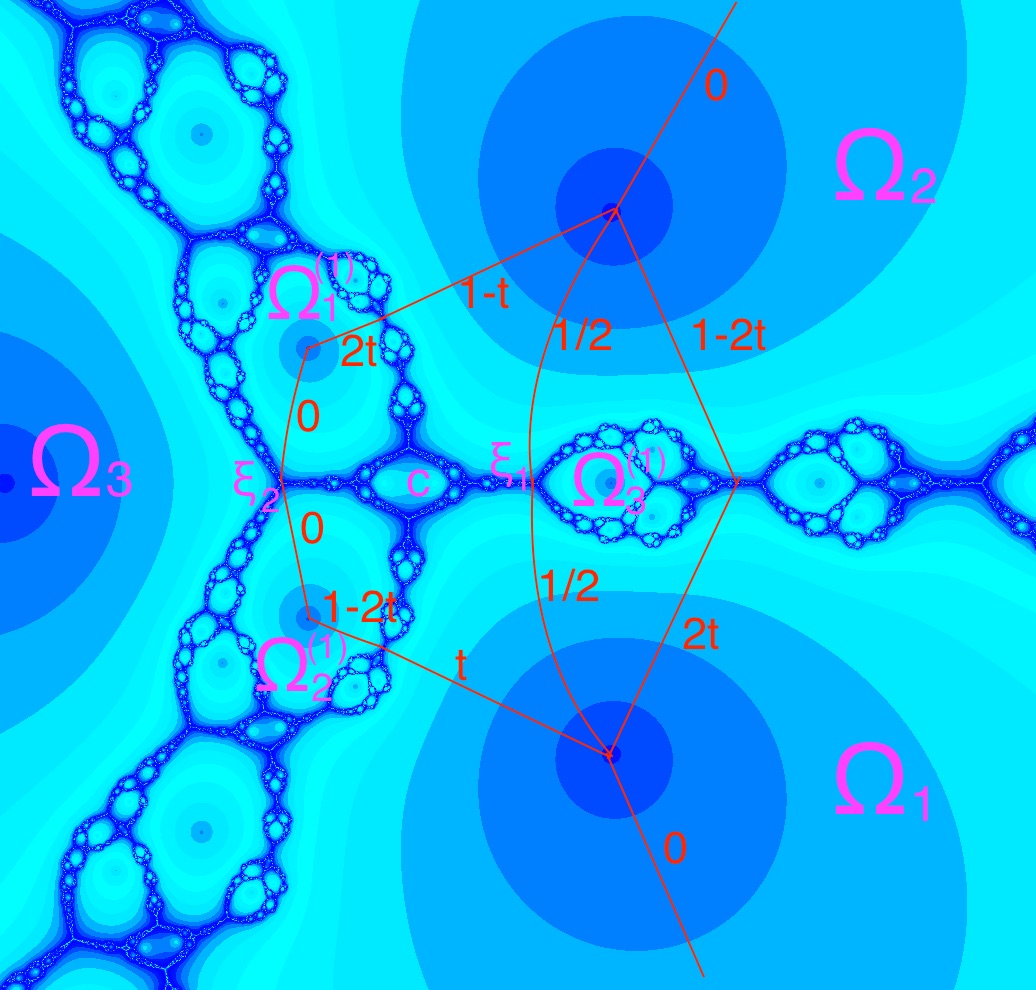}
  \caption{The dynamical plane of the Newton map for the polynomial $z^3/3-z^2/2+1$. Set $\theta=1/2$. Then $2\theta=1\in\Theta$. The internal rays $I_1^{(1)}(0)$ and $I_2^{(1)}(0)$ land at a common point. The angle $t\not\in\Theta$ but $2t\in\Theta$. In this section, we continue to use this example in the subsequent figures.}
  \label{fig:cubic}
\end{figure}

Let $\g(0,1/2)$ be as in Lemma \ref{lem:cubic} (4). Then the complement of $\g(0,1/2)$ in $\widehat{\mathbb{C}}$ contains two components. Denote by $D$ the one that is disjoint with $\Omega_3$. It follows from Lemma \ref{lem:cubic} (4) that $\Omega_3^{(1)}\subset D$.

By Lemma \ref{lem:cubic} (2), we can choose a rational angle $\theta\in(0,1,2)$ satisfying
\begin{itemize}
\item [(i)] $\theta\not\in \Theta,$ but $2\theta\in\Theta$,
\item [(ii)] there exists $k\geq 1$ such that $\eta:=2^k\theta\in(1/2,1)$,  and
\item [(iii)] the orbit of the landing point of $I_1(\theta)$ avoids $c$ and $\infty$.
\end{itemize}
Define
$$\mathcal{L}:=I_3(0)\cup I_3(1/2)\cup I_1(0)\cup I_1(\theta)\cup I_2(0)\cup I_2(1-\theta)\cup I_2^{(1)}(1-2\theta)\cup I_2^{(1)}(0)\cup I_1^{(1)}(0)\cup I_1^{(1)}(2\theta).$$
Then Lemma \ref{lem:cubic} (3) implies that $\mathcal{L}$ is a connected graph. Moreover,  $\widehat{\C}\setminus\mathcal{\mathcal{L}}$ has three components. We label $W$ the one disjoint with $\Omega_3$. In the remaining two components, we label $W_-$ the one intersecting with $\Omega_1$ and label $W_+$ the one intersecting with $\Omega_2$ (see Figure \ref{fig:part}). By Lemma \ref{lem:cubic} (5), it immediately follows that $D\cup\overline{\Omega}_3^{(1)}\subseteq W$ and $c\in W\setminus \ov{D}$. In particular, $\xi_1\in W$. Moreover, we have $I_3(3/4)\subseteq W_-$ and $I_3(1/4)\subseteq W_+$.
\begin{figure}[h]
  \centering
    \includegraphics[width=.5\linewidth]{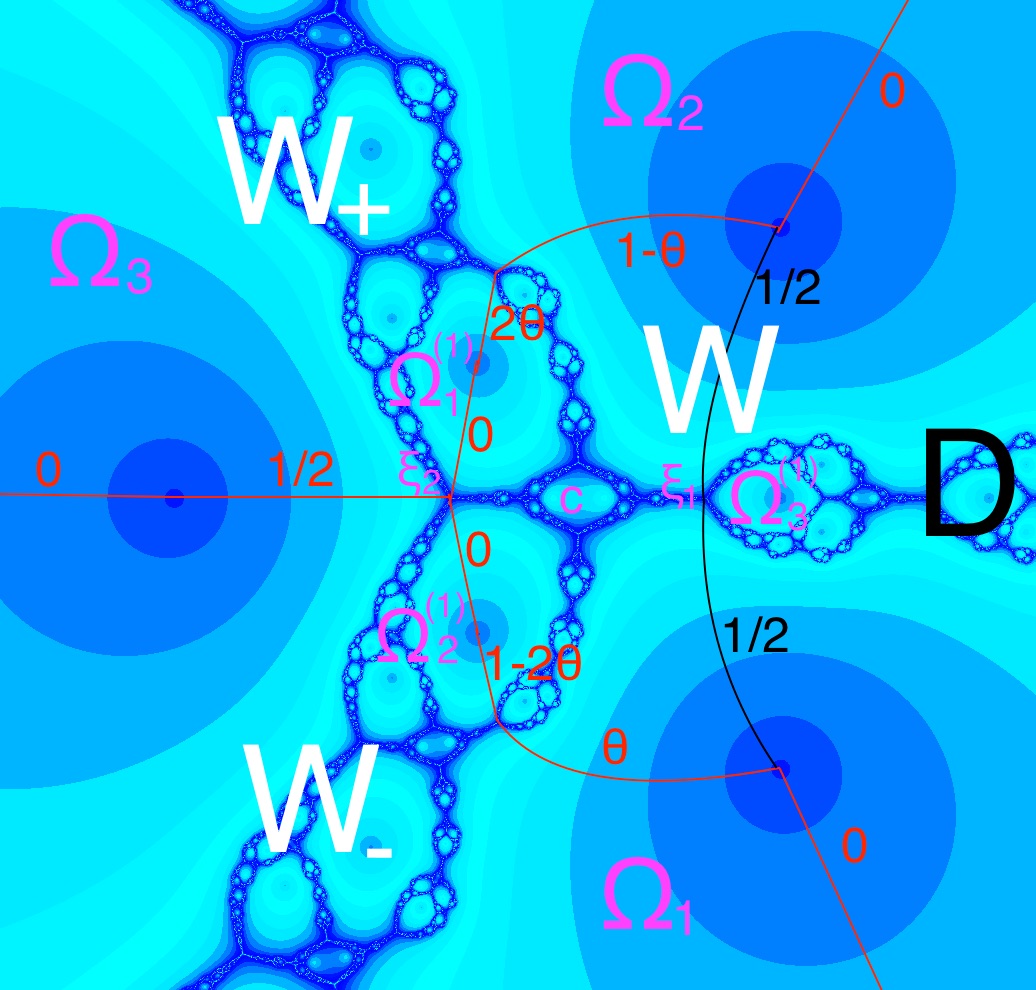}
  \caption{The curve $\mathcal{L}$ consists of the indicated internal rays except $I_1(1/2)$ and $I_2(1/2)$. The boundary of $D$ consists of $I_1(0), I_1(1/2), I_2(1/2)$ and $I_2(0)$.}\label{fig:part}
\end{figure}

Now consider the components of $f^{-1}(\Omega_1^{(1)})$ and $f^{-1}(\Omega_2^{(1)})$. Note that $f^{-1}(\Omega_2^{(1)})$ has a component whose boundary contains the landing point of $I_1((1+\theta)/2)$. Since $I_1((1+\theta)/2)\subset D$, this component is also contained in $D$. Hence it does not contain $c$ since $c\in W\setminus \ov{D}$. Note the landing points of $I_3(1/4)$ and $I_3(3/4)$ are contained in the boundaries of the two remaining components of $f^{-1}(\Omega_2^{(1)})$, respectively. We denote by $\Omega_2^{(2)}$ the component whose boundary contains the landing point of $I_3(3/4)$. Then $I_2^{(2)}(0)$ and $I_3(3/4)$ land at a common point. Moreover, $\Omega_2^{(2)}\subset W_-$ since $I_3(3/4)\subseteq W_-$. It follows that $c\not\in\Omega_2^{(2)}$.
By Lemma \ref{lem:cubic} (3),  we have $I_1(\theta)$ and $I_2^{(1)}(1-2\theta)$ land at a common point. It follows that $I_1(\theta/2)$ and $I_{2}^{(2)}(1-2\theta)$ lands at a common point since $I_1(\theta/2)\subseteq W_-$. Similarly, denote $\Omega_1^{(2)}$ the component of $f^{-1}(\Omega_1^{(1)})$ contained in $W_+$. Then $c\not\in\Omega_1^{(2)}$. Moreover, $I_1^{(2)}(0)$ and $I_3(1/4)$ land at a common point and $I_1^{(2)}(2\theta)$ and $I_2(1-\theta/2)$ land at a common point.
Define the Jordan curve
\begin{eqnarray}\label{eq:11}
\mathcal{C}&:=&I_3(1/4)\cup I_3(3/4)\cup I_2^{(2)}(0)\cup I_2^{(2)}(1-2\theta)\cup I_1(\theta/2)\cup I_1(\eta)\\ \nonumber
&&\cup I_2(1-\eta)\cup I_2(1-\theta/2)\cup I_1^{(2)}(2\theta)\cup I_1^{(2)}(0),
\end{eqnarray}
\begin{figure}[h]
  \centering
 \includegraphics[width=.5\linewidth]{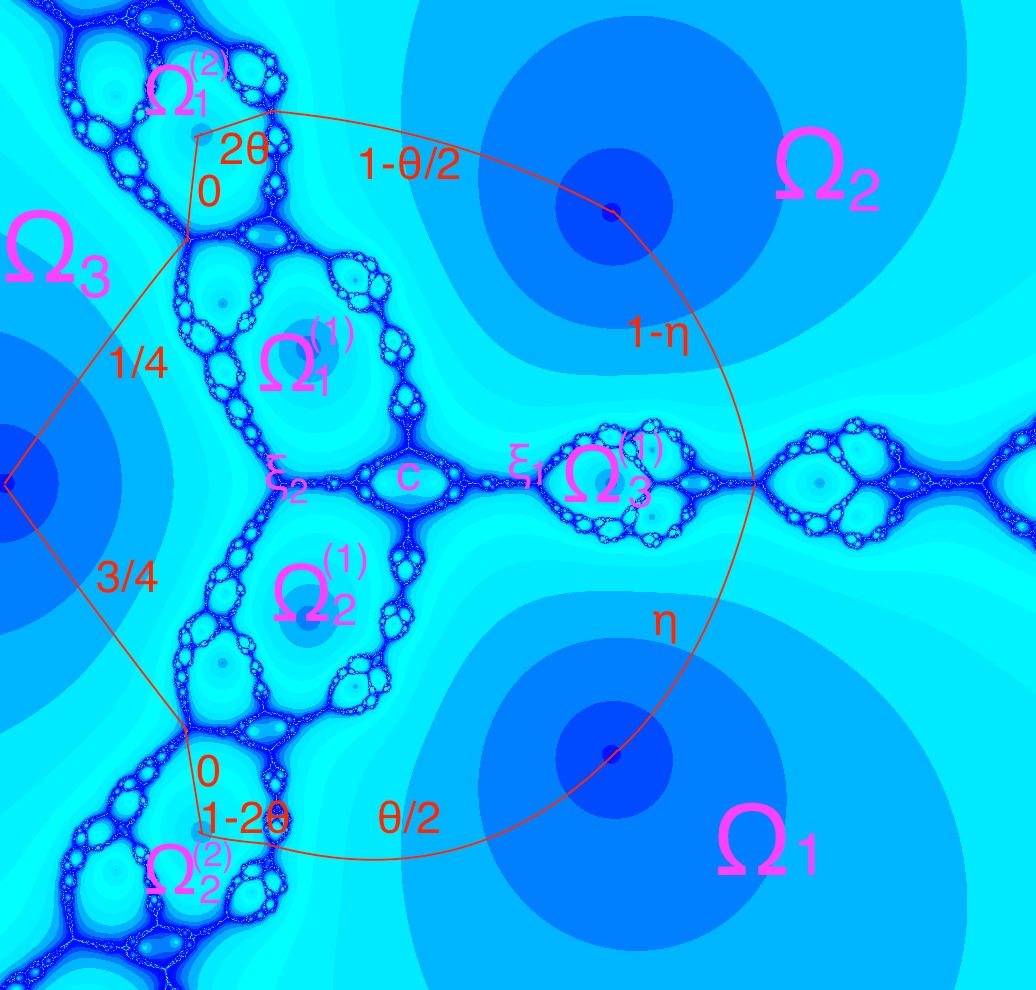}
  \caption{The curve $\mathcal{C}$ consists of the indicated internal rays}
  \label{fig:curve}
\end{figure}

We show that the critical point $c$ is not in the iterations of $\mathcal{C}$ and separated by $\mathcal{C}$ from $\infty$. More precisely:
\begin{lemma}\label{lem:cubic-curve}
Let $\mathcal{C}$ be as above. Then the following hold.
\begin{enumerate}
\item The orbit of any Julia point in $\mathcal{C}$ is disjoint with the critical points of $f$.
\item Denote by $V$ the bounded component of  $\widehat{\mathbb{C}}\setminus\mathcal{C}$. Then
$$\ov{\Omega}_1^{(1)}\cup\ov{\Omega}_2^{(1)}\cup\ov{\Omega}_3^{(1)}\cup\{\xi_1,\xi_2,c\}\subset V$$
\end{enumerate}
\end{lemma}
 \begin{proof}
The Julia points in $\mathcal{C}$ are the landing points of $I_3(1/4),I_3(3/4), I_1(\theta/2),I_1(\eta)$ and $I_2(1-\theta/2)$. By the choice of $\theta$, the orbits of the landing points of $I_1(\theta/2),I_1(\eta),I_2(1-\theta/2)$ are away from $c$. Since $c\in W\setminus\{\infty,\xi_2\}\subseteq \widehat{\C}\setminus \ov{\Omega}_3$, it follows that $c\not\in \partial{\Omega}_3 $, and hence the orbits of the landing points of $I_3(1/4)$ and $I_3(3/4)$ are disjoint with $c$. Then statement (1) holds.

The statement (2) follows immediately from the construction of $\mathcal{C}$ and Lemma \ref{lem:cubic} (4),(5).
 \end{proof}

 Since $\theta$ is rational, there is a positive integer $k>1$ such that the graph
 $$G:=\bigcup_{j=0}^kf^j(\mathcal{C})$$
 is invariant.
 Lemma \ref{lem:cubic-curve} immediately implies that $c\not\in G$. Moreover, obviously our graph $G$ is distinct from the Newton graphs of $f$.

\subsection{Cut angles for quartic Newton maps}\label{sec:quartic}
In this subsection, we generalize part of results in \cite[Section 3]{Roesch08} from cubic case to quartic case.
Throughout this subsection, we assume that $f\in\mathrm{NM}_4$ has degree $2$ in the immediate basin of each root, equivalently, each such immediate basin has a unique critical point, counted with multiplicity.

Let $r_1,r_2,r_3$ and $r_4$ be the roots of $f$ and denote by $\Omega_1,\Omega_2,\Omega_3$ and $\Omega_4$ the corresponding immediate basins. Then there exist $1\le i<j \le 4$ such that $\partial\Omega_i\cap\partial\Omega_j$ contains a pole. Hence the internal rays $I_i(1/2)$ and $I_j(1/2)$ land at common point. We say that $f$ is of \emph{separable type} if there exist $1\le i<j\le 4$ such that $I_i(1/2)$ and $I_j(1/2)$ land at a common point and each component of $\widehat{\C}\setminus\g(0,1/2)$ contains a pole of $f$, where $\g(0,1/2)=I_i(0)\cup I_i(1/2)\cup I_j(0)\cup I_j(1/2)$.

If $f$ is not of separable type, we can choose $1\le i<j\le 4$ such that $I_i(1/2)$ and $I_j(1/2)$ land at a common point, but a component $D$ of $\widehat{\C}\setminus\g(0,1/2)$ does not contain a pole of $f$. Relabeling the roots of $f$, we set $i=1$ and $j=2$. Furthermore, we can set  $I_1(\theta)\in D$ if and only if $\theta\in(1/2,1)$. Hence $I_2(\theta')\in D$ if and only if $\theta'\in(0,1/2)$. We now consider the cut angles in $\Omega_1$. An angle $\theta\in\mathbb{R}/\mathbb{Z}$ is a \emph{cut angle} in $\Omega_1$ if there exists $\theta'\in\mathbb{R}/\mathbb{Z}$ such that $I_1(\theta)$ and $I_2(\theta')$ land at a common point. If $\theta$ is a cut angle in $\Omega_1$, then the corresponding $\theta'=1-\theta$.
Denote by $\Theta$ the set of all cut angles in $\Omega_1$ and set
$$\alpha:=\inf\{\theta:\theta\in\Theta\},$$
where $\inf$ is obtained under the order by identifying $\mathbb{R}/\mathbb{Z}$ with $(0,1]$. Since $\widehat{\C}\setminus \ov{D}$ contains $\Omega_3\cup\Omega_4$, it follows that $\alpha>0$. By the locally connectivity of $\partial\Omega_1$ and $\partial\Omega_2$, we have $\alpha\in\Theta$ and $\Theta$ is a closed set in $\R/\Z$.

Now we state some properties of the cut angles. Since we are interesting in hyperbolic maps, see Section \ref{sec:bdd}, we further assume that $f$ is hyperbolic in the following result.

\begin{proposition}\label{pro:quartic-angle}
 Let $f$ be hyperbolic and not of separable type. With the above notations, the following hold.
\begin{enumerate}
\item For any $\theta\in\Theta$, $(\theta+1)/2\in\Theta$.
\item Let $\theta$ be a periodic angle. If the orbit of $\theta$ belongs to $(\alpha,1)$, then $\theta\in\Theta$.
\item The angles $\alpha\in (0,1/2)$ and $\theta_n:=1-1/(2^n-1)\in \Theta$ for all large $n$.
\end{enumerate}
\end{proposition}
\begin{proof}
For statement (1), since $(\theta+1)/2>1/2$, the internal rays $I_1((\theta+1)/2)\subseteq D$. Suppose $(\theta+1)/2\not\in \Theta$. Since $f(I_1((\theta+1)/2))=I_1(\theta)$ and $\theta\in\Theta$, there exists a component $\Omega_2^{(1)}$ of $f^{-1}(\Omega_2)$ disjoint with $\Omega_2$ such that $\Omega_2^{(1)}$ contains the landing point of $I_1((\theta+1)/2)$. Note that $f$ is hyperbolic and hence the landing point of $I_2(1/2)$ is not a critical point. It follows that $\overline{\Omega}_2^{(1)}\subseteq D$. Hence $D$ contains a pole of $f$. It contradicts to the choice of $D$.

To prove statement (2), let $p$ be the periodic of the angle $\theta$. Under the assumptions of $f$, the unique fixed angle is $0$. It follows that $p>1$. Define
$$\g(0,\alpha):=I_1(0)\cup I_1(\alpha)\cup I_2(0)\cup I_2(1-\alpha).$$
Since $\alpha\leq 1/2$, there exists a component of $\widehat{\C}\setminus \g(0,\alpha)$ containing $D$. Denote this component by $W$. It follows that the only possible pole of $f$ contained in $W$ is the common landing point of $I_1(1/2)$ and $I_2(1/2)$. Hence the only component of $f^{-1}(\Omega_1)$ (resp. $f^{-1}(\Omega_2)$) intersecting with $W$ is $\Omega_1$ (resp. $\Omega_2$) itself.

For each $0\leq i\leq p$, denote by $z_i$ the landing point of $I_1(2^i\theta)$, and by $w_i$ the landing point of $I_2(2^i(1-\theta))=I_2(1-2^i\theta)$. Since $\theta$ is $p$-periodic, the points $z_0,\ldots,z_{p-1}$ (resp. $w_0,\ldots,w_{p-1}$) are pairwise disjoint and $z_0=z_p$ (resp. $w_0=w_p$). Moreover, the assumption of $\theta$ implies that $z_0,\ldots,z_{p-1}, w_0,\ldots,w_{p-1}\in W$. Suppose $\theta\not\in\Theta$. Then $z_0\not=w_0$. As $\Theta$ is closed, we can choose an arc $\ell_0$ in $W\setminus \{I_1(t)\cup I_1(1-t):t\in\Theta\}$ joining the points $z_0=z_p$ and $w_0=w_p$ such that $\ell_0$ is disjoint with $\Omega_1\cup\Omega_2$. Let $\ell_1$ be the lift of $\ell_0$ based at $z_{p-1}$. By the choice of $\ell_0$, we have $\ell_1$ belongs to $W\setminus \{I_1(t)\cup I_1(1-t):t\in\Theta\}$ and is disjoint with $\Omega_1\cup\Omega_2$. Note that the endpoint of $\ell_1$ is on the boundary of a preimage of $\Omega_2$. By the previous paragraph, this component is $\Omega_2$ itself. Note also that $w_{p-1}$ is the unique preimage of $w_p$ on $\partial\Omega_2$ such that $w_{p-1}$ and $z_{p-1}$ are in the same component of $W\setminus(I_1(1/2)\cup I_2(1/2))$.  Hence the endpoints of $\ell_1$ is $w_{p-1}$. Inductively, for each $m\geq1$, we get an arc $\ell_{mp}\subseteq W$ joining $z_0$ and $w_0$ which is a lift of $\ell_0$ by $f^{pm}$. Choose $\ell_0$ such that it does not intersect  the closure of the forward orbits of the critical points of $f$. Since $f$ is hyperbolic, and hence it is uniformly expanding near the Julia set $J_f$. It follows that the length of $\ell_{mp}$ converges to $0$ as $m\to \infty$. Then $z_0=w_0$, a contradiction. Hence $\theta\in\Theta$ and statement (2) follows.

Now we prove statement $(3)$. Note that $\alpha\in(0,1/2]$. Suppose, on the contrary, that $\alpha=1/2$. According to statement (1), the angles $1-1/2^n\in \Theta$ for all $n\geq1$. Choose an angle $\eta\in\Theta$ close to $1$ and define $\g(0,\eta):=I_1(0)\cup I_1(\eta)\cup I_2(1-\eta)\cup I_2(0)$. Let $D_\eta$ be component of $\widehat{\C}\setminus\g(\eta)$ contained in $D$. Choose $\eta$ sufficiently close to $1$ such that $D_\eta$ contains no critical value of $f$. Since $\alpha=1/2$, then $I_1(\eta/2)$ and $I_2(1-\eta/2)$ land at distinct points. Denote by $\Omega_1^{(1)}$ the component of $f^{-1}(\Omega_1)$ such that $I_2(1-\eta/2)$ and $I_1^{(1)}(\eta)$ land at a common point and denote by $\Omega_2^{(1)}$ the component of $f^{-1}(\Omega_2)$ such that $I_1(\eta/2)$ and $I_2^{(1)}(1-\eta)$ land at a common point. Since $f$ is hyperbolic, its Julia set contains no critical point. It follows that there exists a component $D_\eta'$ of $f^{-1}(D_\eta)$ whose boundary contains the arc $$I_1(1/2)\cup I_2(1/2)\cup I_1(\eta/2)\cup I_2(1-\eta/2)\cup I_2^{(1)}(1-\eta)\cup I_1^{(1)}(\eta).$$
Note that the two arcs $I_1(\eta/2)\cup I_2^{(1)}(1-\eta)$ and $I_2(1-\eta/2)\cup I_1^{(1)}(\eta)$ are disjoint and mapped to the same arc $I_1(\eta)\cup I_2(1-\eta)$ under $f$. Then the proper map $f:D_\eta'\to D_\eta$ has degree at least $2$. It implies that $D_\eta'$ contains at least one critical point, and hence $D_\eta$ contains a critical value. It contradicts to the choice of $D_\eta$.

For the second part of statement (3), we first note that $\theta_n$ has periodic $n$. Now we apply statement (2). We only need to show $2^i\theta\in(\alpha,1)$ for $0\le i\le n-1$. If $0\le i<n-1$, we have  $2^i\theta_n=1-\frac{2^i}{2^n-1}\in (1/2,1)$. For $i=n-1$, we have $2^{n-1}\theta_n=\frac{1}{2}(1-\frac{1}{2^n-1})$. Since $\alpha<1/2$, it follows that $2^{n-1}\theta_n\in(\alpha,1)$ for sufficiently large  $n$.
\end{proof}

\section{Proof of Theorem \ref{main}}\label{sec:pf main}
The goal of this section is to prove Theorem \ref{main}. We define B\"{o}ttcher coordinates in the deformations in Section \ref{sec:Bottcher} and prove the convergence of B\"{o}ttcher coordinates, see Proposition \ref{lem:Bottcher-converge}. To do that, we introduce the dynamically weak Carath\'eodory topology, see Definition \ref{def:dyn-conv}.  In Section \ref{sec:rays}, we use the B\"{o}ttcher coordinates in the deformations to define the corresponding internal rays, and then show a convergence result on these rays, see Proposition \ref{lem:perturb-ray}.  Finally, we prove Theorem \ref{main} in Section \ref{pf-main}.

\subsection{Perturbation of B\"{o}ttcher coordinates}\label{sec:Bottcher}
Let $f=H_f\hat f\in\overline{\mathrm{NM}}_d$ with $\deg\hat f\ge 2$ and denote by $\Omega_{\hat f}$ the union of basins of the roots of $\hat f$. Let $\mathcal{U}$ be a forward invariant finite subset of components of $\Omega_{\hat f}$. Recall that $\hat f$ is postcritically finite in $\mathcal{U}$ if the critical points in any $U\in\mathcal{U}$ have finite orbits. For such $\hat f$ and $\mathcal{U}$, one can choose a system of B\"{o}ttcher coordinates $\{\phi_U:U\to\D\}_{U\in\mathcal{U}}$ satisfying
 \[\text{$\phi_{\hf(U)}\circ \hf\circ \phi_U^{-1}(z)=z^{d_U}, z\in\D$, where $d_U:={\rm deg}(\hf|_U)$.}\]

Let $\{f_n\}_{n\ge 1}$ be a sequence in $\mathrm{NM}_d$ such that $f_n$ converges to $f$ semi-algebraically. Since $\hat f$ is \pf in $\UUU$, the Fatou components in $\UUU$ are disjoint with the holes of $f$. Indeed, a possible hole of $f$ is either $\infty$ or an attracting fixed point, see Section \ref{sub:Newton}. Hence, by Lemma \ref{lem:fatou}, for $U\in\mathcal{U}$, its center $u$ belongs to a component $U_n$ of $\Omega_{f_n}$ for all large $n$. We call such $U_n$ the \emph{deformation} of $U$ at $f_n$.  In this subsection, under natural assumption, we define a B\"{o}ttcher coordinate $\phi_{U_n}$ on the deformations $U_n$ of $U$ and show a convergence result of $\phi_{U_n}$.

We first recall the definition of \emph{weak Carath\'eodory topologies} on set of pointed sets and set of holomorphic functions, respectively. Let $\mathcal{V}$ be a set of open simply-connected pointed sets $(V,v)$ in $\mathbb{C}$. The weak Carath\'eodory topology on $\mathcal{V}$ is defined by the following convergence: $(V_n, v_n)$ converges to $(V,v)$ if and only if (i) $v_n$ converges to $v$;  and (ii) for any compact $K\subset V$, we have $K\subset V_n$ for all large $n$.
Denote $\mathcal{G}$ the set of holomorphic functions defined on $(V, v)\in\mathcal{V}$. Then the \emph{weak Carath\'eodory topology} on $\mathcal{G}$ is defined as follows. Let $g:(V, v)\to\mathbb{C}$ and $g_n:(V_n, v_n)\to\mathbb{C}$ be functions in $\mathcal{G}$. We say $g_n$ converges to $g$ if (i) $(V_n, v_n)$ converges to $(V, v)$ in $V$; and (ii) $g_n$ converges to $g$ uniformly on compact subsets of $V$ for all large $n$.

Back to Newton maps, let $\{f_n\}_{n\ge 1}$ and $f$ be as above. For $U\in\mathcal{U}$, denote $u:=\phi^{-1}_U(0)$ the center of $U$. We use the following definition.
\begin{definition}\label{def:dyn-conv}
We say $f_n$ converges to $f$ in $\mathcal{U}$ under the \emph{dynamically weak Carath\'eodory topology} if for each $U\in\mathcal{U}$, there exists a point $u_n$ in the deformation $U_n$ of $U$ such that
\begin{enumerate}
\item $f_n:(U_n, u_n)\to\mathbb{C}$ converges to $\hat f:(U, u)\to\mathbb{C}$ under weak Carath\'eodory topology;
\item  the point $u_n$ is  (pre)periodic under $f_n$ with the same (pre)period as that of $u$; and
\item  the local degrees $\deg_u f=\deg_{u_n} f_n$.
\end {enumerate}
\end{definition}
We call such $u_n$ (if exists) a \emph{center} of $U_n$, and call $(U_n,u_n)$ the \emph{deformation} of $(U,u)$ at $f_n$. To abuse notations, we denote the set of pointed sets $(U,u)$ with $U\in\mathcal{U}$ also by $\mathcal{U}$. Set
$$\mathcal{U}_n:=\{(U_n,u_n):(U_n,u_n)\text{ is the deformation of }(U,u)\in\mathcal{U}\}.$$
It may happen that the set $U_n$ contains several distinct centers:
\begin{remark}
If a critical point $c$ of $\hat{f}$ is contained in the boundaries of distinct $(U,u)$ and $(U',u')$ in $\mathcal{U}$, it is possible that $U_n$ coincide with $U_n'$ and it contains the critical point of $f_n$ perturbed from $c$ (see Figure \ref{fig:1}). In this case, both $u_n$ and $u_n'$ are centers of $U_n=U_n'$, and hence $f_n$ is not postcritically finite in the union of $U_n$ with $(U_n,u_n)\in\mathcal{U}_n$.
\end{remark}
\begin{figure}[h]
  \begin{minipage}[b]{0.45\textwidth}
    \includegraphics[width=\textwidth]{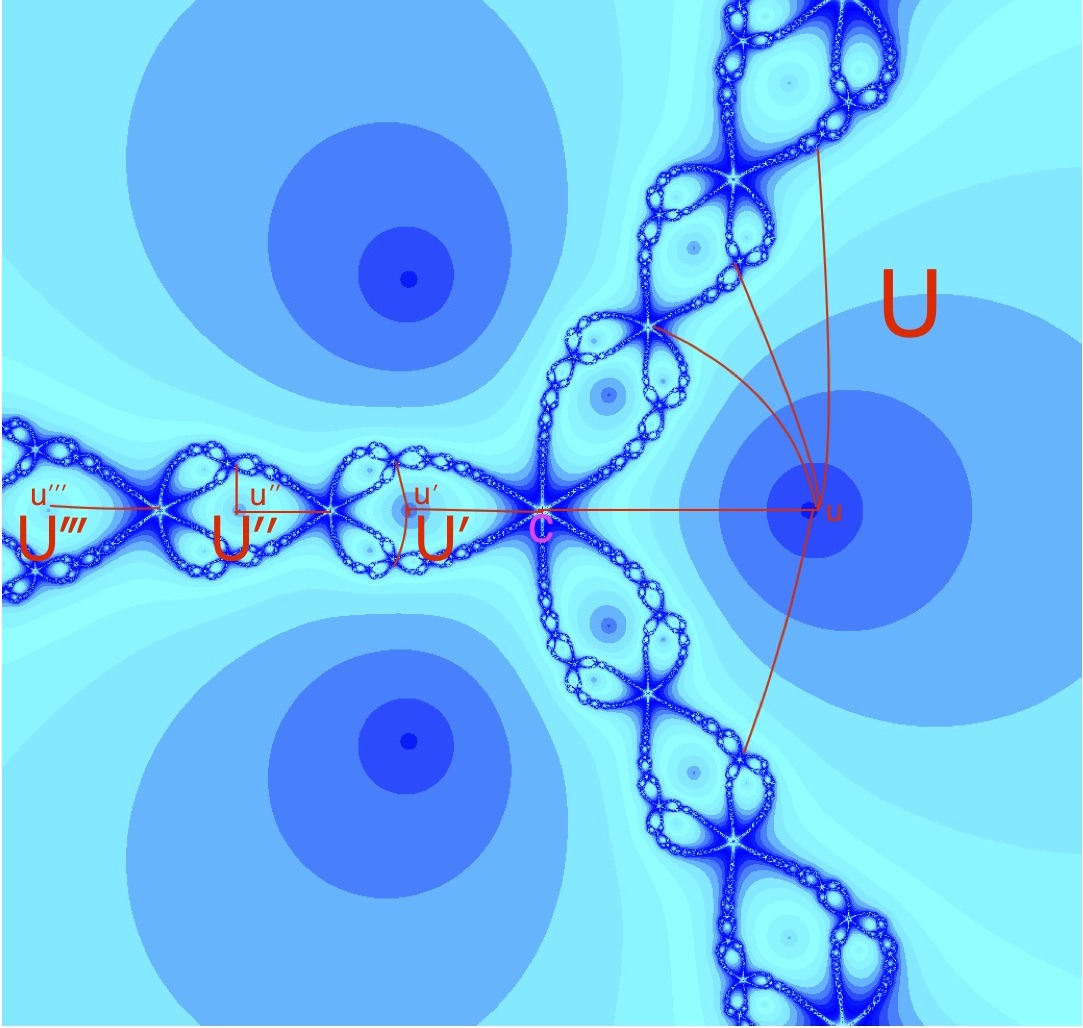}
  \end{minipage}
  \begin{minipage}[b]{0.45\textwidth}
    \includegraphics[width=0.99\textwidth]{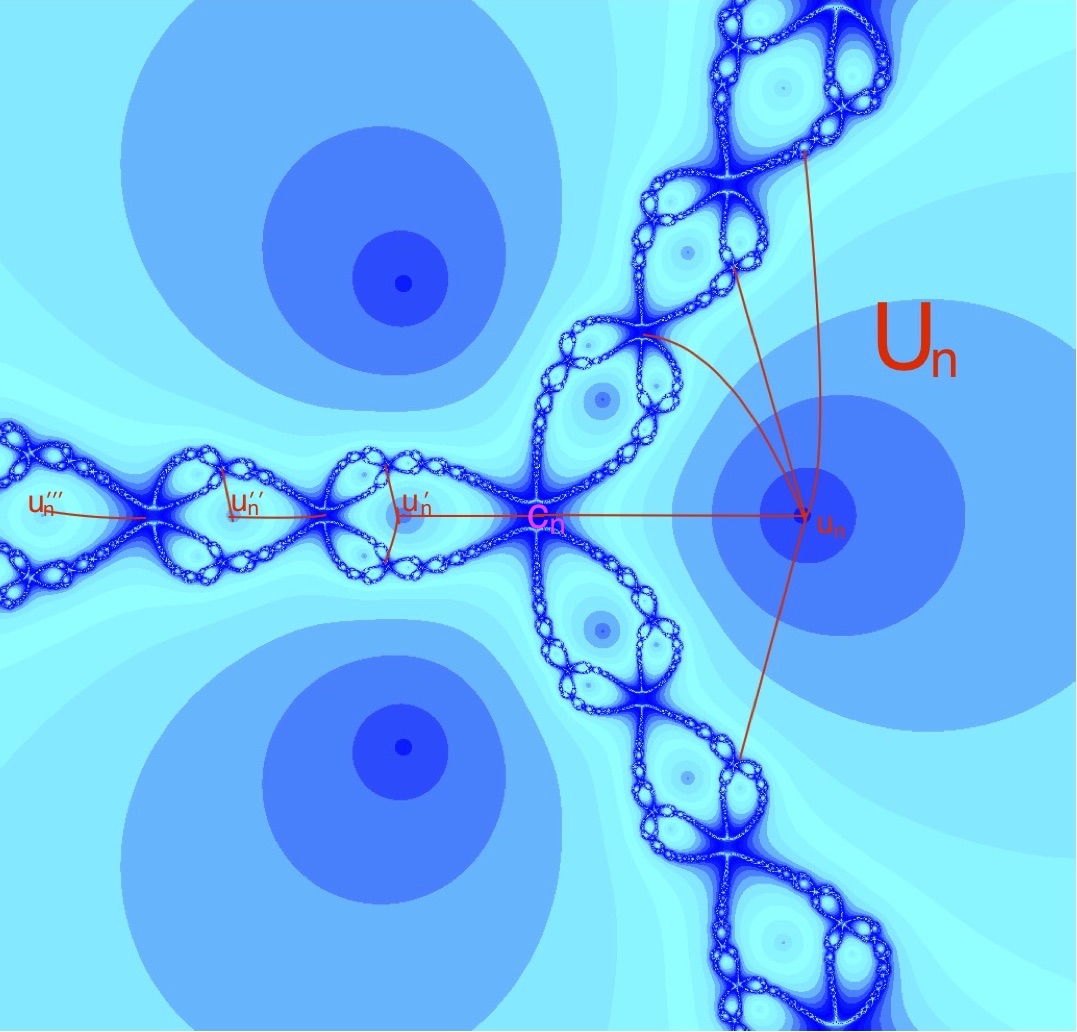}
  \end{minipage}
  \caption{Left: the dynamical plane of the Newton map $f$ for the polynomial $z^3-1$. The letters indicate Fatou components $U, U', U''$ and $U'''$ with centers $u,u',u''$ and $u'''$, respectively. The arcs indicate internal rays. The critical point $c=0$ is contained in $\partial U\cap\partial U'$. Right: the dynamical plane of the Newton map $f_n$ for the polynomial $z^3+z/n-1$ with indicated Fatou component $U_n$. The critical point $c_n\in U_n$. The points $u_n, u_n', u_n''$ and $u_n'''$ are all in $U_n$ and centers of $U_n$. The set $(U_n,u_n)$ is the deformation of $(U,u)$; the set $(U_n,u'_n)$ is the deformation of $(U',u')$; the set $(U_n,u''_n)$ is the deformation of $(U'',u'')$; and the set $(U_n,u'''_n)$ is the deformation of $(U''',u''')$. The corresponding rays in $U_n$ either land on $\partial U_n$ or terminate at the iterated preimages of $c_n$.}
  \label{fig:1}
\end{figure}

The following result states a natural sufficient condition for the convergence under the dynamically weak Carath\'eodory topology, which we use repeatedly in Section \ref{sec:bdd}.
\begin{lemma}\label{lem:criterion}
Let $\{f_n\}_{n\ge 1}$ be a sequence in $\mathrm{NM}_d$ such that $f_n$ converges to $f=H_f\hat f$ semi-algebraically with ${\rm deg}(\hat f)\geq2$. Let $\mathcal{U}$ be a forward invariant finite subset of components of $\Omega_{\hat f}$ and suppose $\hat f$ is \pf in $\UUU$. If every immediate basin of roots in $\UUU$ has degree $2$, and every non immediate basin in $\UUU$ has degree $1$, then $f_n$ converges to $f$ in $\UUU$ under the dynamically weak Carath\'eodory topology.
\end{lemma}
\begin{proof}
Since every $U\in\UUU$ avoids the poles of $f$, by Lemmas \ref{lem:periodic-point} and \ref{lem:fatou}, we have the following: for every $(U,u)\in \UUU$,
\begin{enumerate}
\item there exists a unique (pre)periodic point $u_n$ of $f_n$ near $u$ with the same (pre)period as that of $u$ such that $u_n\to u$. In particular, if $U$ is the immediate basin of a root of $\hat f$, then $u_n$ is the root of $f_n$ contained in $U_n$ (the deformation of $U$ at $f_n$);
\item any compact subset of $U$ is contained in $U_n$ for all large $n$;
\item given $k\geq1$, the $k$-th derivative $f_n^{(k)}(u_n)$ converges to $\hat f^{(k)}(u)$ as $n\to\infty$.
\end{enumerate}
The statements (1) and (2) imply that $f_n:(U_n, u_n)\to\mathbb{C}$ converges to $\hat f:(U, u)\to\mathbb{C}$ under weak Carath\'eodory topology. Together with statement (3), we have that if $\hat f(U)=U$ (immediate basin), then $f_n'(u_n)=0$ but $f_n^{''}(u_n)\not=0$, i.e., $\deg_{u_n} f_n=2=\deg_u\hat f$; and if $\hat f(U)\not=U$ (non immediate basin), then $f_n'(u_n)\not=0$, i.e., $\deg_{u_n} f_n=1=\deg_u\hat f$. It completes the proof
\end{proof}

From now on, we assume $f_n$ converges to $f$ in $\mathcal{U}$ under the dynamically weak Carath\'eodory topology. Since $\hat f$ is postcritically finite in $\mathcal{U}$, by Lemma \ref{lem:semi-convergence}, we have the following straight forward result and omit the proof.
\begin{lemma}\label{rem:simply-converge}
If $(U_n,u_n)$ is the deformation of $(U,u)\in\mathcal{U}$ at $f_n$, then
$(f_n(U_n),f_n(u_n))$ is the deformation of $(\hf(U),\hf(u))$.
\end{lemma}

The above lemma suggests that for each $(U_n,u_n)\in\mathcal{U}_n$, we have a B\"{o}ttcher coordinate $\phi_{(U_n,u_n)}$ near $u_n$ such that
\begin{equation}\label{eq:15}
\phi_{(U_n,u_n)}(z)^{d_U}=\phi_{(f_n(U_n),f_n(u_n))}\circ f_n(z)
\end{equation}
for $z$ near $u_n$, and that
 \begin{equation}\label{eq:16}
 \phi_{(U_n,u_n)}'(u_n)\to \phi_{(U,u)}'(u)\text{ as }n\to\infty.
  \end{equation}
Given any $r\in(0,1)$, the map $\phi_{(U_n,u_n)}$ extends conformally until meeting an iterated preimage of critical points of $f_n$. Then there exists a maximum $r_n\le 1$ such that
$\psi_{(U_n,u_n)}:=\phi_{(U_n,u_n)}^{-1}: \mathbb{D}_{r_n}\to U_n$ is conformal.

Denote $\psi_{(U,u)}$ the inverse of $\phi_{(U,u)}$. The following result asserts that $\psi_{(U_n,u_n)}$ converges to $\psi_{(U,u)}$ locally uniformly on $\mathbb{D}$.

\begin{proposition}\label{lem:Bottcher-converge}
For $(U,u)\in\mathcal{U}$, let $(U_n,u_n)\in\mathcal{U}_n$ be the deformation of $(U,u)$. Then $\psi_{(U_n,u_n)}$ converges to $\psi_{(U,u)}$ locally uniformly on $\mathbb{D}$.
\end{proposition}
\begin{proof}
It is sufficient to show that for any given $0<r<1$, the maps $\psi_{(U_n,u_n)}$ uniformly converge to $\psi_{(U,u)}$ on $\ov{\D}_r:=\{z:|z|\leq r\}$.
We first assume $\hat{f}(U,u)=(U,u)$, i.e, $\hat f(u)=u$.  Then $f_n(U_n,u_n)=(U_n,u_n)$ for all large $n$.
Given any $r\in(0,1)$, let $r_1\in(r,1)$ and denote $U(r_1):=\psi_{(U,u)}(\D_{r_1})$. By Lemma \ref{lem:fatou}, we have $\overline{U}(r_1)\subseteq U_n$ for all large $n$. Since $U(r_1)$ contains no critical point of $\hat f$ except $u$, the B\"ottcher coordinate $\phi_{(U_n,u_n)}$ extends to $U(r_1)$.

Note that $\{\phi_{(U_n,u_n)}\}_{n\ge 1}$ is a normal family on $U(r_1)$. Let $\phi_{(U_{n_k},u_{n_k})}$ be any converging subsequence and denote the limit by $\phi$. By Equation \eqref{eq:15} and Lemma \ref{lem:semi-convergence}, it follows that $z^{d_U}\circ\phi=\phi\circ \hat f$ on $U(r_1)$. Hence, $\phi$ is a B\"{o}ttcher coordinate of $\hat f$ on $U$. According to the convergence \eqref{eq:16}, we obtain $\phi=\phi_{(U,u)}$. By the arbitrariness of $\phi_{(U_{n_k},u_{n_k})}$, the sequence $\phi_{(U_n,u_n)}$ uniformly converges to $\phi_{(U,u)}$ in $U(r_1)$. As a consequence, the image domain $\phi_{(U_n,u_n)}(U(r_1))$ contains $\D_r$ for all large $n$, and $\psi_{(U_n,u_n)}$ uniformly converges to $\psi_{(U,u)}$ on $\D_r$.

In the general case, by inductively using the argument above, we can prove the conclusion.
\end{proof}

\subsection{Perturbation of internal rays}\label{sec:rays}
In previous subsection, we perturb a B\"{o}ttcher coordinate in $(U,u)\in\mathcal{U}$ to obtain a B\"{o}ttcher coordinate $\phi_{(U_n,u_n)}$ in $(U_n,u_n)\in\mathcal{U}_n$. In this subsection, we use the inverse map $\psi_{(U_n,u_n)}$ to define the internal rays in $(U_n,u_n)$ and prove a convergence result on internal rays.

Now we define internal rays of $f_n$ in $(U_n,u_n)$ as follows. For each $\theta\in\R/\Z$, let $r_\theta$ be the maximal radius such that $\psi_{(U_n,u_n)}$ extends along $(0,r)e^{2\pi i\theta}$. If $r_\theta<1$, then arc $\psi_{(U_n,u_n)}((0,r)e^{2\pi i\theta})$ terminates at an iterated preiamge of critical points of $f_n$, and if $r_\theta=1$, the arc $\psi_{(U_n,u_n)}((0,r)e^{2\pi i\theta})$ accumulates, factually lands on $\partial U_n$. In the latter case, we call
$$I_{(U_n,u_n)}(\theta):=\psi_{(U_n,u_n)}([0,1]e^{2\pi i\theta})$$
 the \emph{landed internal ray in $(U_n,u_n)$ of angle $\theta$}. Note that $f_n$ sends a landed internal ray of $(U_n,u_n)$ to a landed internal ray of $(f(U_n),f(u_n))$. Also, since $U_n$ may contains more than one centers, it may possess several groups of landed interval rays. In this case, each such ray starts from a center of $U_n$ and rays from distinct groups are disjoint (see Figure \ref{fig:1}).

\begin{remark}\label{rem:list}
We list two simple cases that $I_{(U_n,u_n)}(\theta)$ is a landed internal ray for large $n$. 
\begin{enumerate}
\item The domain $U$ is the immediate basin of a root of $\hat f$ and $\deg(f_n|_{U_n})=\deg(\hf|_U)$ for all large $n$. In this case, $I_{(U_n,u_n)}(\theta)$ is a landed internal ray for all $\theta\in\mathbb{R}/\mathbb{Z}$  since $U_n$ contains no critical points other than $u_n$.
\item  The orbit of the landing point of $I_{(U,u)}(\theta)$ is away from the critical points of $\hat f$ and  $\deg(f_n|_{\Omega_n})=\deg(\hat f|_\Omega)$ for all large $n$, where $\Omega$ is the immediate basin of a root such that $U$ is an iterated preimage of $\Omega$ and $\Omega_n$ is the deformation of $\Omega$
\end{enumerate}
\end{remark}

The following result asserts that the internal rays of eventually periodic angles converge.
\begin{proposition}\label{lem:perturb-ray}
For $(U,u)\in\mathcal{U}$, suppose that the internal ray $I_{(U,u)}(\theta)$ of angle $\theta$ lands at an eventually repelling periodic point. For all large $n$, let $I_{(U_n,u_n)}(\theta)$ be the landed internal ray in $(U_n,u_n)\in\mathcal{U}_n$ of angle $\theta$. Then $I_{(U_n,u_n)}(\theta)\to I_{(U,u)}(\theta)$ as $n\to\infty$.
\end{proposition}
\begin{proof}
 To ease notations, we write $I(\theta),I_n(\theta),\psi,\psi_n$ for $I_{(U,u)}(\theta), I_{(U_n,u_n)}(\theta),\psi_{(U,u)}\psi_{(U_n,u_n)}$, respectively. Set  $\de:={\rm deg}(\hat{f}|_U)$ and let $z_0$ be the landing point of $I(\theta)$. It is sufficient to show that, given any $\eta>0$, for all large $n$, we have $d_H(I(\theta),I_n(\theta))<\eta$, where $d_H$ is the Hausdorff metric.

First assume that $I(\theta)$ is periodic of period $p\ge 1$. Then $u$ is a (super)attracting fixed point of $\hat{f}$. Define $D_\epsilon:=\{z\in\widehat{\mathbb{C}}:\rho(z,z_0)<\epsilon\}$, where $\rho$ is the spherical metric. Shrinking $\epsilon$ if necessary, we may assume $\hf|_{D_\epsilon}$ is injective and $\ov{D}_\epsilon\subseteq \hat f^p(D_\epsilon)$. We claim that for any sufficiently large $n$ and any component $D'$  of $f^{-p}_n(D_\epsilon)$, either $\ov{D}'\subseteq D_\epsilon$ or $\ov{D}'\subseteq \widehat{\C}\setminus \ov{D}_\epsilon$. Indeed, if $p>1$, the landing point $z_0$ of $I(\theta)$ is not a hole of $f^p$, see Section \ref{deg-rat} and \cite[Lemma 2.2]{DeMarco05}. It follows from Lemma \ref{lem:semi-convergence} that $f_n^p$ converges uniformly to $\hat{f}^p$ near $z_0$, and hence $f_n^p|_{D_\epsilon}$ is injective and $\ov{D}_\epsilon\subseteq f_n^p(D_\epsilon)$ for all large $n$. Then the claim holds. Now we consider the case that $p=1$. Then $z_0=\infty$. If $z_0=\infty$ is not a hole of $f$, the claim follows by previous argument. If $z_0=\infty$ is a hole of $f$, then $f_n$ fails to converge uniformly to $\hat{f}$ near $\infty$. In this case, we prove the claim by contradiction. Suppose that the claim fails. Then there exists a subsequence, denoted also by $\{f_n\}$, such that for each $f_n$, there exists a component $D_n'$ of $f^{-1}_n(D_\epsilon)$ satisfying $\ov{D}'_n\cap \partial D_\epsilon\not=\emptyset$. Choose a point $w_n\in \ov{D}_n'\cap \partial D_\epsilon$. Passing to subsequence if necessary, we may assume $w_n\to w$. Then $w\in\partial D_\epsilon$. By Lemma \ref{lem:semi-convergence}, the sequence $f_n$ converges uniformly to $\hat f$ on $\partial D_\epsilon$. It follows that $f_n(w_n)\to \hat f(w)$ as $n\to\infty$.  Note that $\overline{D}_\epsilon\subseteq f(D_\epsilon)$. Then $\hat f(\partial D_\epsilon)\cap\ov{D}_\epsilon=\emptyset$. It follows that $f(w)\not\in\ov{D}_\epsilon$. However,  $f_n(w_n)\in f_n(\ov{D}_n')=\ov{D}_\epsilon$, which implies $f(w)\in \ov{D}_\epsilon$. It is a contradiction. Therefore, the claim holds.

Since $I(\theta)$ lands at $z_0$, there exists $0<r<1$ such that $\psi((r,1)e^{2\pi i\theta})\subseteq U\cap D_\epsilon$. Pick $0<s<1$ such that $s^{\de^p}>r$.
 Then the segment $\psi([s^{\de^p},s]e^{2\pi i \theta})\subseteq I(\theta)$ belongs to $U\cap D_\epsilon$. It follows from Proposition \ref{lem:Bottcher-converge} that for all large $n$,
 \begin{eqnarray}
 d_H(\psi_n\large([0,s]e^{2\pi i\theta}\large),\psi\large([0,s]e^{2\pi i\theta}\large))<\epsilon.\label{eq:10}
 \end{eqnarray}
 Define $\g_{n,0}:[0,1]\to\psi_n([s^{\de^p},s]e^{2\pi i\theta})$ be an arc such that $\g_{n,0}(0)=\psi_n(s^{\de^p}e^{2\pi i\theta})$ and $\g_{n,0}(1)=\psi_n(s e^{2\pi i\theta})$.
 Then $\g_{n,0}([0,1])\subseteq D_\epsilon\cap U_n.$ Note that $f_n^p(\g_{n,0}(1))=\g_{n,0}(0)$. Lift $\g_{n,0}$ to an arc $\g_{n,1}$ based at $\g_{n,0}(1)$. Inductively, we obtain a sequence of arcs $\g_{n,k}$ such that  $\g_{n,k+1}$ is a lift by $f_n$ of $\g_{n,k}$ based at the endpoint of $\g_{n,k}$ which is not in $\g_{n,k-1}$ .

Now we claim that for sufficiently large $n$, the arc $\g_{n,k}\subset D_\epsilon$. We prove the claim by induction on $k$. The claim holds for $k=0$ by the definition of $\g_{n,0}$.
Suppose that for $k\geq0$, the arc $\g_{n,k}\subseteq D$. Since $\g_{n,k+1}$ is a preimage of $\g_{n,k}$ under $f_n$, there exists a component $D'$ of $f_n^{-1}(D_\epsilon)$ containing $\g_{n,k+1}$. Since the intersection point of $\g_{n,k+1}\subseteq D'$ and $\g_{n,k}\subseteq D_\epsilon$ belongs to $D_\epsilon$, it follows that $D'\cap D_\epsilon\not=\emptyset$. By the previous claim, we have $D'\subseteq D_\epsilon$, and hence $\g_{n,k+1}\subseteq D_\epsilon$, which completes the induction.

 Note that for all large $n$,
 $$I_n(\theta)=\psi_n([0,s]e^{2\pi i\theta})\bigcup(\cup_{k\geq0}\g_{n,k})\cup\{z_n\},$$
 where $z_n$ is the landing point of $I_n(\theta)$.
According to estimate \eqref{eq:10} and the fact that $\g_{n,k}\subseteq D_\epsilon$, we have $d_H(I(\theta),I_n(\theta))<\epsilon$. By choosing $\epsilon<\eta$, we prove the proposition under the periodicity assumption.

In the strictly preperiodic case, we set $(V,v):=\hat{f}(U,u)$ and $I(\theta')=\hat f(I(\theta))$. Let $(V_n,v_n)$ be the deformation of $(V,v)$ with $f_n(U_n,u_n)=(V_n,v_n)$. Inductively, it is sufficient to prove $d_H(I(\theta),I_n(\theta))<\epsilon$ under the assumption that $\lim_{n\to\infty}d_H(I(\theta'),I_n(\theta'))=0$.

Define $D_\epsilon$ as above. By Proposition \ref{lem:Bottcher-converge}, there exists $0<s<1$ such that for all large $n$,
\[\text{$d_H(\psi_n([0,s]e^{2\pi it}),\psi([0,s]e^{2\pi it}))<\epsilon$ and $\psi_n(se^{2\pi it})\in D_\epsilon$}.\]
Denote by $L'_n:=\psi_n([s^\de,1]e^{2\pi\theta'})$ and $L':=\psi([s^\de,1]e^{2\pi\theta'})$, respectively. Since $I_n(\theta')\to I(\theta')$, we have $L'_n$ and $L'$ are contained in $\hat f(D_\epsilon)$ for large $n$. Since $I_n(\theta)$ is a landed internal ray for all large $n$, there is a lift $L_n$ of $L'_n$ based at the point $\psi_n(se^{2\pi i\theta})$. Denote by $L$ the lift of $L'$ based at the point $\psi(se^{2\pi i\theta})$. Note that in this case we have $z_0\not\in\mathrm{Hole}(f)$. Then $f_n$ converges uniformly to $\hat{f}$ on $D_\epsilon$. Thus $f(D_\epsilon)\subset f_n(D_{2\epsilon})$ for sufficiently large $n$. Hence we have $L_n\subset D_{2\epsilon}$ and $L\subseteq D_{2\epsilon}$. Note $I(\theta)=\psi([0,s]e^{2\pi it})\cup L$ and $I_n(\theta)=\psi_n([0,s]e^{2\pi it})\cup L_n$. It follow that $d_H(I(t),I_n(t))<2\epsilon$. Choose $\epsilon<\eta/2$. This completes the proof.
\end{proof}

\subsection{Proof of Theorem \ref{main}}\label{pf-main}
Now we prove the Theorem \ref{main}.
 \begin{proof}[Proof of Theorem \ref{main}]
By Lemma \ref{lem:perturb-ray}, we have that for each $((U,u),t)\in\mathcal{V}\times T$, the internal rays  $I_{(U_n,u_n)}(t)\to I_{(U,u)}(t)$ as $n\to\infty$. It follows immediately that $\G_n\to\G$ as $n\to\infty$. It remains to check that $\G_n$ is homeomorphic to $\G$ for large $n$. It is sufficient to show that for any $((U,u),t)$ and $((U',u'),t')$ in $\mathcal{V}\times T$, the rays $I_{(U,u)}(\theta)$ and $I_{(U',u')}(t')$ land at a common point if and only if $I_{(U_n,u_n)}(\theta)$ and $I_{(U_n',u_n')}(\theta')$ land at a common point for all large $n$. It immediately follows from Lemmas \ref{lem:periodic-point} and \ref{lem:perturb-ray} since the orbits of the Julia points in $\G$ are away from the critical points of $\hat{f}$.
\end{proof}

\section{The boundedness of hyperbolic components}\label{sec:bdd}

In this section, we aim to prove Theorem \ref{thm:bdd hyp}. In Section \ref{classification}, we classify the hyperbolic components into several types and state known boundedness results. Section \ref{key-lam} contains two key lemmas for the proof of Theorem \ref{thm:bdd hyp}: one concerns the orbit of a critical point and the limit of an attracting cycle; the other one concerns the combinatorial property of the limit function. Then we prove Theorem \ref{thm:bdd hyp} in Section \ref{sec:pf-bdd}.

\subsection{Classification of hyperbolic components and known results }\label{classification}

Let $f\in\mathrm{NM}_4$ be the Newton map of the quartic polynomial $P$. Then the finite fixed points of $f$ are the zeros of $P$ and the critical points of $f$ are the zeros of $P$ and zeros of $P''$. Hence zeros of $P$ are the superattracting fixed points of $f$. We call any other (super)attracting cycles of $f$ is a \emph{free (super)attracting cycle}. Then any free (super)attracting cycle has period at least $2$. Moreover, we say a critical point $c$ of $f$ is \emph{additional} if $P''(c)=0$. Hence $f$ has two additional critical points, counted with multiplicity. According to the orbits of the additional critical points, the hyperbolic components in the moduli space $\mathrm{nm}_4:=\mathrm{NM}_4/\mathrm{Aut}(\mathbb{C})$ belong to the following seven types, see \cite{Nie18}.

\textbf{Type A. Adjacent critical points,} with both additional critical points in the same component of the immediate basin of a free (super)attracting cycle.\par
\textbf{Type B. Bitransitive,} with both additional critical points in the immediate basin of a free (super)attracting period cycle, but they do not lie in the same component.\par
\textbf{Type C. Capture,} with one additional critical point in the immediate basin of a free (super)attracting cycle, the other additional critical point in the basin but not the immediate basin of this cycle.\par
\textbf{Type D. Disjoint (super)attracting orbits,} with both additional critical points in the immediate basins of two distinct free (super)attracting cycles.\par
\textbf{Type IE. Immediate Escape,} with some additional critical point in the immediate basin of a superattracting fixed point.\par
\textbf{Type FE1. One Future Escape,} with one additional critical point in the basin (but not immediate basin) of a superattracting fixed point, while the other additional critical point is in the immediate basin of a free (super)attracting cycle.\par
\textbf{Type FE2. Two Future Escape,} with both additional critical points in the basins (but not immediate basins) of one or two superattracting fixed points.\par

The above classification is an analogy of that for quadratic rational maps \cite{Milnor93} and for cubic polynomials \cite{Milnor09}.

Recall that a hyperbolic component in $\mathrm{nm}_4$ is bounded if it has a compact closure in $\mathrm{nm}_4$. Since the type D hyperbolic components have algebraic boundaries, an arithmetic argument shows that such components are bounded:
\begin{proposition}\cite[Main Theorem]{Nie18}\label{D to bdd}
The hyperbolic components of type D in $nm_4$ are bounded.
\end{proposition}
In contrast, all hyperbolic components of type IE are unbounded.
\begin{proposition}\cite[Theorem 1.4]{Nie18}\label{IE to unbdd}
Let $\mathcal{H}\subset\mathrm{nm}_4$ be a hyperbolic component. If $\mathcal{H}$ is of type IE, then $\mathcal{H}$ is unbounded in $\mathrm{nm}_4$.
\end{proposition}

In the remainder of this section, we give more bounded hyperbolic components in $\mathrm{nm}_4$. In fact, we show the condition in Proposition \ref{IE to unbdd} is also necessary.

\subsection{Key lemmas}\label{key-lam}
To prove Theorem \ref{thm:bdd hyp}, we need two key lemmas.

Let $\{f_n\}\subset\mathrm{NM}_4$ be a sequence converging to $f=H_f\hat f\in\overline{\mathrm{NM}}_4$ such that $\mathrm{Hole}(f)=\{\infty\}$ and $\deg\hat f=3$. Then $f_n$ has a unique additional critical point $c_n$ converging to $\infty$ as $n\to\infty$. If all $f_n$s are in a same hyperbolic component, denote $\mathcal{O}_n$ an attracting cycle of $f_n$ of period $m\ge 2$. Our first lemma states the orbit of $c_n$ and the limit of $\mathcal{O}_n$.

\begin{lemma}\label{lem:key}
Let $f_n, f, c_n$ and $\mathcal{O}_n$ be as above. Then the following hold.
\begin{enumerate}
\item Given any $k\geq0$ and small $\epsilon>0$, the points $c_n, f_n(c_n),\ldots,f^k_n(c_n)$ are in the $\epsilon$-neighborhood of $\infty$ for all large $n$;
\item Suppose $\mathcal{O}_n$ converges to $\mathcal{O}$ as $n\to\infty$. Then $\mathcal{O}\not=\{\infty\}$.
\end{enumerate}
\end{lemma}
\begin{proof}
Denote by $r_{1,n},r_{2,n},r_{3,n}$ and $r_{4,n}$ the roots of $f_n$. Since $\mathrm{Hole}(f)=\{\infty\}$ and $\deg\hat f=3$, we may assume $r_{4,n}\to\infty$, as $n\to\infty$, and for $1\le i\le 3$, the point $r_{i,n}$ is outside the $\epsilon$-neighborhood of $\infty$ for all large $n$. Define $M_n(z):=r_{4,n}z$ and let $g_n=M_n^{-1}\circ f_n\circ M_n$. Then $g_n\in\mathrm{NM}_4$ with roots at $r_{1,n}/r_{4,n},r_{2,n}/r_{4,n},r_{3,n}/r_{4,n}$ and $1$. Let $g=H_g\hat g$ be the degenerate Newton map of the polynomial $z^3(z-1)$, Then $g_n$ locally uniformly converges to $\hat g$ away from $\mathrm{Hole}(g)=\{0\}$. Note that $\hat g$ has a critical point at $\tilde c=1/2$ and $\tilde{c}$ is attracted to the attracting fixed point $0$. Given any $k\geq0$, the point $\tilde c$ is not in ${\rm Hole}(g^k)=\cup_{i=0}^{k-1}\hat g^{-i}(0)$ and $|\hat g^k(\tilde c)|>\epsilon_0$ for some positive number $\epsilon_0$. By Lemma \ref{lem:semi-convergence}, we have $|g_n^k(\tilde c_n)|>\epsilon_0$ for all large $n$. Note that for the maps $f_n$, we have $f^k_n(c_n)=M_n(g^k_n(\tilde{c}_n))$. It follows that $|f^k_n(c_n)|>r_{4,n}\epsilon_0$. Thus, statement (1) follows.

For statement (2), write $\mathcal{O}_n=\{w_n^{(0)},\cdots, w_n^{(m-1)}\}$. Suppose to the contrary that $\mathcal{O}=\{\infty\}$. Then all $w_i^{(0)}$s converge to $\infty$. In the following argument, we may pass to subsequences if necessary to obtain limits. Relabeling the indices, we may assume $w_n^{(i)}/w_n^{(0)}$ dose not converge to $0$ for all $0\le i\le m-1$. Write $L_n(z)=w_n^{(0)}z$. Then $\mathcal{O}'_n:=\{1,w_n^{(1)}/w_n^{(0)},\cdots, w_n^{(n-1)}/w_n^{(0)}\}$ is an attracting cycle of $h_n:=L_n^{-1}\circ f_n\circ L_n$. Denote by $\mathcal{O}'$ the limit of $\mathcal{O}_n'$. Then $0\not\in\mathcal{O}'$. Assume that $h_n\to h=H_h\hat h$. Note $\mathrm{Hole}(h)\subset\{0,\infty\}$ and $1\le\deg \hat h\le 2$.
If $\deg\hat h=1$, then $\hat h$ has an attracting fixed point at $0$ and a repelling fixed point at $\infty$. Moreover, $\mathrm{Hole}(h^j)=\mathrm{Hole}(h)\subset\{0,\infty\}$ for all $j\ge 1$. It follows that $\mathcal{O}'\cap\mathrm{Hole}(h)=\emptyset$. Then by Lemma \ref{lem:limit cycle}, the set $\mathcal{O}'$ is a non-repelling cycle of $\hat h$. Note $1\in\mathcal{O}'$ is not a fixed point of $\hat h$. It is a contradiction since $\hat h$ has degree $1$ and hence all the periodic points of $\hat h$ are fixed points. If $\deg\hat h=2$, then $\mathrm{Hole}(h)=\{0\}$. Moreover, $\hat h$ has an attracting fixed point at $0$, a superattracting fixed point at the limit $r$ of $r_4^{(n)}/w_n^{(0)}$ and a repelling fixed point at $\infty$. Since $0\not\in\mathcal{O}'$, then $\mathcal{O}'\cap\mathrm{Hole}(f)=\emptyset$. By Lemma \ref{lem:limit cycle}, the set $\mathcal{O}'$ is a non-repelling cycle of $\hat h$ of period at least $2$. It follows that $\hat h$ has at least $3$ non-repelling cycles: two (super)attracting fixed points $0$ and $r$, and one non-repelling cycle $\mathcal{O}'$. It contradicts to the Fatou-Shishikura inequality (see \cite{Shishikura87}) which asserts that $\hat f$ has at most $2$ non-repelling cycles. Therefore, we have $\mathcal{O}\not=\{\infty\}$ and the conclusion follows.
\end{proof}

Recall that a quartic Newton map $f\in\mathrm{NM}_4$ is of separable type if there exist two distinct immediate basins $\Omega_i$ and $\Omega_j$ of roots of $f$ such that the corresponding internal rays $I_1(1/2)$ and $I_2(1/2)$ land at a pole and the curve $I_1(0)\cup I_1(1/2)\cup I_2(1/2)\cup I_2(0)$ separates the remaining poles of $f$. We say a hyperbolic component $\mathcal{H}$ of $\mathrm{nm}_4$ is of \emph{separable type} if each element in $\mathcal{H}$ is of separable type, equivalently, there is an element of separable type in $\mathcal{H}$. Otherwise, we say $\mathcal{H}$ is of \emph{inseparable type}.

Our next key lemma asserts that a non type IE hyperbolic component is of inseparable type under extra assumption on its lift.

\begin{lemma}\label{lem:exclude}
Let $\mathcal{H}\subset\mathrm{nm}_4$ be a non type IE hyperbolic component and let $\widetilde{\mathcal{H}}\subset\mathrm{NM}_4$ be a lift of $\mathcal{H}$. Suppose $\{f_n\}\subset\widetilde{\mathcal{H}}$ such that $f_n$ converges to $f=H_f\hat{f}\in\overline{\mathrm{NM}}_4$ with $\mathrm{Hole}(f)=\{\infty\}$ and ${\deg}(\hat{f})=3$.
Then $\mathcal{H}$ is of inseparable type and all poles of $\hat{f}$ are simple.
\end{lemma}

\begin{proof}
By the assumptions, $\hat f$ has three superattracting fixed points, denoted by $r_1,r_2$ and $r_3$ respectively. Let $\Omega_1,\Omega_2$ and $\Omega_3$ be the corresponding immediate basins. Moreover, the reduction $\hat f$ has a unique critical point $c$ with $c\not\in\cup_{i=1}^3\Omega_i$. By Lemma \ref{lem:criterion}, the sequence $f_n,$ converges $f$ in $\{\Omega_1,\Omega_2,\Omega_3\}$ under the dynamically weak Carath\'eodory topology.
Relabeling $r_i$s, we may assume that there exists a pole of $\hat f$ in the intersection $\partial\Omega_1\cap\partial\Omega_2$.
For $1\le i\le 3$, denote $(\Omega_{i,n}, r_{i,n})$ the deformation of $(\Omega_i,r_i)$ at $f_n$. Then $\partial\Omega_{1,n}\cap\partial\Omega_{2,n}$ contains a pole of $f_n$. Denote by $\Omega_{4,n}$ the remaining immediate basin of $f_n$ at the superattracting fixed point $r_{4,n}$. Then $r_{4,n}\to\infty$, as $n\to\infty$.

On the contrary, we assume $\HHH$ is of separable type. Consider the internal rays in $\Omega_{1,n}$ and $\Omega_{2,n}$ and set $\g_n(0,1/2):=I_{1,n}(0)\cup I_{1,n}(1/2)\cup I_{2,n}(0)\cup I_{2,n}(1/2)$. Then each component of $\widehat{\C}\setminus \g_n(0,1/2)$ contains a pole of $f_n$, and hence contains $\Omega_{3,n}$ or $\Omega_{4,n}$  We denote by $D_n$ the one containing $\Omega_{4,n}$, and assume that $I_{1,n}(\theta)\subseteq D_n$ if and only if $\theta\in(1/2,1)$.

Since $\Omega_{4,n}\subset D_n$, there exists a minimal $k\geq2$ such that the landing point $z_n$ of $I_{2,n}(1/2^k)$ is not in $\partial\Omega_{1,n}$. Let $\Omega_{1,n}^{(1)}$ be the component of $f_n^{-1}(\Omega_{1,n})$ such that $z_n\in\partial\Omega_{1,n}^{(1)}$. Then $\Omega_{1,n}^{(1)}\not=\Omega_{1,n}$ and $\Omega_{1,n}^{(1)}\subseteq D_n$. Note that $\Omega_{1,n}^{(1)}$ contains no critical point. For otherwise, $\partial \Omega_{1,n}^{(1)}$, hence $D_n$, would contain two poles of $f_n$, which is impossible. Then $\partial\Omega_{1,n}^{(1)}$ contains a unique pole of $f_n$, which coincides with the one on $\partial\Omega_{4,n}$. Set $I_{1,n}^{(1)}(t)$ the internal ray in $\Omega_{1,n}^{(1)}$ landing at $z_n$.
By Remark \ref{rem:list} (1) and Proposition \ref{lem:perturb-ray}, the landing point $z_n$ of $I_{2,n}(1/2^k)$ converges to the landing point $z$ of $I_2(1/2^k)$. Note that the pole of $f_n$ in $\partial \Omega_{1,n}\cap \partial\Omega_{2,n}$ (resp. $\partial\Omega_{3,n}$) converges to the pole of $\hat f$ in $\partial\Omega_1\cap\partial\Omega_2$ (resp. $\partial\Omega_3$). Thus, the pole of $f_n$ in $\partial\Omega_{4,n}\cap\partial\Omega_{1,n}^{(1)}$ converges to $\infty$, as $n\to\infty$. For otherwise, these poles converge to poles of $\hat f$, contradicting to $\deg\hat f =3$. Similarly, the center of $\Omega_{1,n}^{(1)}$ converges to $\infty$. Then, passing to subsequences if necessary, the arcs $I_{1,n}^{(1)}(t)$ converge to a continuum $\ell$ containing $\infty$ and ${z}$.

Recall that $\psi_{1,n}^{(1)}:\D\to \Omega_{1,n}^{(1)}$ and $\psi_{1,n}:\D\to \Omega_{1,n}^{(1)}$ are the inverses of the B\"{o}ttcher coordinates on $\Omega_{1,n}^{(1)}$ and $\Omega_{1,n}$, respectively.
Let ${q}$ be any point in $\ell\setminus\{\infty\}$. There exists $q_n\in I_{1,n}^{(1)}(t)$ with $q_n\to {q}$. We write $q_n=\psi_{1,n}^{(1)}(s_ne^{2\pi it})$. Since $q\not=\infty$, we have $f_n(q_n)\to \hf({q})$. Note
\[f_n(q_n)=f_n\circ\psi_{1,n}^{(1)}(s_ne^{2\pi it})=\psi_{1,n}(s_ne^{2\pi it})\in I_{1,n}(t).\]
Since $I_{1,n}(t)\to I_{1}(t)$, the point $\hat f({q})$ belongs to $I_{1}(t)$. We claim that $\hat f({q})\in\partial\Omega_1$. Otherwise, ${q}$ belongs to either $ \Omega_1$ or the other component $\Omega_1^{(1)}$ of $\hf^{-1}( \Omega_{1})$. Note that $\Omega_1^{(1)}\cap D_n=\emptyset$ for large $n$. By Lemma \ref{lem:fatou}, we have $q_n\not\in\Omega_{1,n}^{(1)}$. It is a contradiction. By this claim, any point in $\ell\setminus \{\infty\}$ maps under $\hf$ to the landing point of $I_1(t)$. It is impossible. Thus, $\mathcal{H}$ is of inseparable type.

Now we show all poles of $\hat f$ are simple. Let $\Theta$ be the set of angles $\theta$ such that $I_{\Omega_{1,n}}(\theta)$ and $I_{\Omega_{2,n}}(1-\theta)$ land at a common point. Since $\mathcal{H}$ is of inseparable type, by Proposition \ref{pro:quartic-angle} (3), there exists a periodic angle $\theta\in\Theta$ of period at least $2$. According to Remark \ref{rem:list} (1) and Lemma \ref{lem:perturb-ray}, the internal rays $I_{1}(\theta)$ and $I_{2}(1-\theta)$ land at a common point. By Lemma \ref{lem:cubic} (6), any pole is not a critical point. It follows that all poles of $\hat f$ are simple.
\end{proof}

\subsection{Proof of Theorem \ref{thm:bdd hyp}}\label{sec:pf-bdd}

In this subsection, we prove Theorem \ref{thm:bdd hyp}. We first recall the statement.
\begin{theorem}\label{thm:bdd hyp-1}
Let $\mathcal{H}\subset\mathrm{nm}_4$ be a hyperbolic component. Then $\mathcal{H}$ is unbounded in $\mathrm{nm}_4$ if and only if $\mathcal{H}$ is of type IE.
\end{theorem}

Before we prove Theorem \ref{thm:bdd hyp-1}, we first state the following lift result.
\begin{lemma}\label{lem:lift}
For $d\ge 3$, let $[g_n]\in\mathrm{nm}_d$ be a sequence such that $[g_n]\to\infty$. Then there exists a sequence $f_{n_i}\in\mathrm{NM}_d$ such that $[f_{n_i}]=[g_{n_i}]$ and $f_{n_i}$ converges to $f=H_f\hat f\in\partial\mathrm{NM}_d$ with $\mathrm{Hole}(f)=\{\infty\}$ and $\deg\hat f\ge 2$. Moreover, if  all $[g_n]$s are contained in a same hyperbolic component in $\mathrm{nm}_d$, then $f_{n_i}$s are contained in a same  hyperbolic component in $\mathrm{NM}_d$.
\end{lemma}
\begin{proof}
Since $[g_n]\to\infty$, there exists a subsequence $g_{n_i}$ such that $g_{n_i}$ converges to an element in $\partial\mathrm{NM}_d$. We first normalize the roots of $g_{n_i}$ by affine maps to obtain a sequence $\tilde{g}_{n_i}\in\mathrm{NM}_d$ such that $0$ and $1$ are two roots of $\tilde{g}_{n_i}$. Note $[\tilde{g}_{n_i}]=[g_{n_i}]$. It follows that $[\tilde{g}_{n_i}]\to\infty$ and hence $\{\tilde{g}_{n_i}\}$ contains a subsequence converging to an element in $\partial\mathrm{NM}_d$. We also denote this subsequence by $\{\tilde{g}_{n_i}\}$. We can further assume all roots of $\tilde{g}_{n_i}$ converge in $\widehat\C$. Denote by $r_{1,n_i},\cdots,r_{d,n_i}$ the roots of  $\tilde{g}_{n_i}$. Choose $1\le m_0< m_1\le d$ such that $|r_{m_0,n_i}-r_{m_1,n_i}|=O(|r_{\ell,n_i}-r_{k,n_i}|)$ for all $1\le \ell< k\le d$ with $r_{\ell,n_i}\not\to\infty$ and $r_{k,n_i}\not\to\infty$, as $n_i\to\infty$.
Define $M_{n_i}(z):=(z-r_{m_1,n_i})/(r_{m_0,n_i}-r_{m_1,n_i})$  and set $f_{n_i}:=M_{n_i}\circ\tilde g_{n_i}\circ M^{-1}_{n_i}$. Then $f_{n_i}$ has roots at $0, 1$ and no roots colliding in $\mathbb{C}$. Then the sequence $f_{n_i}$ is the desired sequence.

Now we claim that each hyperbolic component in $\mathrm{nm}_d$ has a unique lift in $\mathrm{NM}_d$. Indeed, if $[g]\in\mathrm{nm}_d$ is contained in a hyperbolic conponent, let $f\in\mathrm{NM}_d$ satisfy $[f]=[g]$ and $f\not=g$. There exists an affine map $M(z)$ such that $g=M\circ f\circ M^{-1}$. Let $\gamma: [0,1]\to\mathrm{Aut}(\mathbb{C})$ be a curve such that $\gamma(0)=id$ and $\gamma(1)=M$. Note $g_t:=\gamma(t)\circ g\circ(\gamma(t))^{-1}\in\mathrm{NM}_d$ is hyperbolic. It follows that $g_t$ and $g$ are contained in the same hyperbolic component. Note $g_1=f$. Hence $f$ and $g$ are contained in the same hyperbolic component. The claim holds. Thus if $[g_n]$s are contained in a hyperbolic component $\mathcal{H}\subset\mathrm{nm}_d$, the above claim implies immediately that $f_{n_i}$s are contain in  the unique lift $\widetilde{\mathcal{H}}\subset\mathrm{NM}_d$ of $\mathcal{H}$.
\end{proof}

\begin{proof}[Proof of Theorem \ref{thm:bdd hyp-1}]
By Proposition \ref{IE to unbdd}, it is sufficient to show if $\mathcal{H}\subset\mathrm{nm}_4$ is not of type IE, then $\mathcal{H}$ is bounded in $\mathrm{nm}_4$.
The proof goes by contradiction.

Suppose $\mathcal{H}$ is unbounded.
Let $\{[f_n]\}_{n\ge 0}$ be a degenerated sequence in $\mathcal{H}$.  Passing to a subsequence, by Lemma \ref{lem:lift} we can assume that all $f_n$ belong to a hyperbolic component of $\mathrm{NM}_4$ and $f_n\to f=H_f\hat f\in\overline{\mathrm{NM}}_4$ with $\mathrm{Hole}(f)=\{\infty\}$ and $\deg(\hat f)=2$ or $3$. We deduce the contradiction case by case.

\noindent{\bf Case 1: $\deg\hat f=2$.}

Let $(\Omega_1,r_1)$ and $(\Omega_2,r_2)$ be the immediate basins of roots of $\hat f$. By Lemma \ref{lem:criterion}, we have $f_n\to f$ in $\{\Omega_1,\Omega_2\}$ under the dynamically weak Carath\'eodory topology. Denote by $(\Omega_{1,n},r_{1,n})$ and $(\Omega_{2,n},r_{1,n})$ the  deformations of $(\Omega_1,r_1)$ and $(\Omega_2,r_2)$ at $f_n$, respectively. In this case, the Julia set $J_{\hat f}=\partial\Omega_1=\partial\Omega_2$ is a Jordan curve and contains no critical points. Given any rational angle $\theta$, the internal rays $I_1(\theta)$ and $I_2(1-\theta)$  land at a common point. By Remark \ref{rem:list} (1) and Theorem \ref{main}, for all large $n$, the internal rays $I_{1,n}(\theta)$ and $I_{1,n}(1-\theta)$ land at a common point. Since all $f_n$ belong to the same hyperbolic component,  we get that the internal rays $I_{1,0}(t)$ and $I_{2,0}(1-t)$ of $f_0$ land together for all $t\in\Q$. By the density of the rational angles in $\R/\Z$, the boundaries $\partial\Omega_{1,0}$ and $\partial\Omega_{1,0}$ coincide. It follows that $f_1$ is conjugate to $z\mapsto z^2$, which is a contradiction.

\noindent{\bf Case 2: $\deg\hat f=3$.}

In this case, $\hf\in\mathrm{NM}_3$ and its unique additional critical point $c$ is not in the immediate basins of the roots of $\hat f$. For otherwise $f_n$ would possess an additional critical point in the immediate basin of some root, which is a contradiction.

Let $\mathcal{O}_n:=\{z_n^{(0)},\dots,z_n^{(m-1)}\}$ be an attracting cycle of $f_n$ of period $m>1$ and denote by $U(z_n^{(i)})$ the Fatou component of $f_n$ containing $z_n^{(i)}$. Suppose $\mathcal{O}_n$ converges to $\mathcal{O}$ as $n\to\infty$. We claim that either $\mathcal{O}\subseteq\C$, or the critical points of $f_n$ contained in the immediate basin of $\OOO_n$ stay in a compact set in $\C$. Indeed, assume that $\infty\in\mathcal{O}$. 
Suppose $U(z_n^{(i)})$ contains a critical point $c_n$ of $f_n$. By relabeling the index, we can assume that $i=0$. If $z_n^{(0)}\to\infty$, by Lemma \ref{lem:key} (2), there exists an $1\le j\le p-1$ such that $z_n^{(j)}\not\to\infty$. It follows from Lemma \ref{lem:2.8} that the basin $U(z_n^{(j)})$ stays outside a neighborhood of $\infty$ for all large $n$. Since $f_n^{j}(c_n)\in U(z_n^{(j)})$, Lemma \ref{lem:key} (1) implies that $c_n\not\to\infty$. If $z_n^{(0)}\not\to\infty$, some $z_n^{(\ell)}$ with $1\le j\le p-1$ must converge to $\infty$ since $\infty\in\mathcal{O}$. Again by Lemma \ref{lem:2.8}, the basin $U(z_n^{(0)})$ stay outside a neighborhood of $\infty$ for all large $n$. Hence $c_n\not\to\infty$.

Now we proceed our argument according to the type of $\mathcal{H}$.

\noindent{{\bf Case 2.(i)}: $\mathcal{H}$ is of type A,B,C or D.}

Let $\mathcal{O}_n$ be a free (super)attracting cycle of $f_n$ of same period and let $\mathcal{O}$ be the limit of $\mathcal{O}_n$. We claim that $\mathcal{O}\subseteq\mathbb{C}$. Note $f_n$ has an additional critical point converging to $\infty$. By the claim in the previous paragraph, if $\mathcal{H}$ is of type A, B or D, it follows that $\mathcal{O}\subseteq\mathbb{C}$. If $\mathcal{H}$ is of type C, suppose the claim fails. By Lemma \ref{lem:key} (2), there exist periodic points $z_n^{(i)}$ and $z_n^{(j)}$ in $\mathcal{O}_n$ such that $z_n^{(i)}\to \infty$ but $z^{(j)}_n\not\to\infty$. It follows from Lemma \ref{lem:2.8} that the basin $U(z_n^{(j)})$ stays outside a neighborhood of $\infty$ for all large $n$. Denote by $c_n$ the additional critical point of $f_n$ which converges to $\infty$. By the claim in the previous paragraph, $c_n$ is not in the immediate basin of $\mathcal{O}_n$. Notice that $f_n^k(c_n)\in U(z_n^{(j)})$ with some $k$ independent on $n$. It contradicts to Lemma \ref{lem:key} (1). Hence the claim holds.

Since $\mathcal{O}\subseteq \C$, by Lemma \ref{lem:limit cycle}, the set $\mathcal{O}$ is a non-repelling cycle of $\hat f$ of period at least $2$. Then $\hat f$ is \pf in $\Omega_{\hat f}$ and the unique additional critical point of $\hat f$ is not an iterated preimage of $\infty$ under $\hat f$.

Consider the Newton graph $\De_m(\hat f)$ of $\hat f$ at level $m$. By Proposition \ref{thm:WYZ}, for a sufficiently large $m$, there exists a Jordan curve $\g\subseteq \De_m(\hat f)$ such that the orbit $\mathcal{O}$ is contained in the bounded component of $\widehat{\C}\setminus \g$.
Let $\UUU$ be the collection of components of $\Omega_{\hat f}$ intersecting $\De_m(\hat f)$. Then $\hat f(\UUU)\subseteq\UUU$. By Lemma \ref{lem:criterion}, the sequence $f_n$ converges to $f$ in $\UUU$ under the dynamically week Carath\'eodory topology.

 Denote $\delta:=d_H(\infty,\g)$.
By Remark \ref{rem:list} (2) and Theorem \ref{main},  the curve $\g$ is perturbed to a Jordan curve $\g_n\subseteq \De_m(f_n)$ such that  $\mathcal{O}_n$ is contained in the bounded component of $\widehat{\C}\setminus\g_n$ and $d_H(\g_n,\g)<\delta/3$ for all large $n$. Since the immediate basin of $\mathcal{O}_n$ is disjoint with $\De_m(f_n)$ for all $n$, it is contained in the bounded component of $\widehat{\C}\setminus \g_n$.

If $\mathcal{H}$ is of type A or B, the above argument immediately implies that the distance between any additional critical point of $f_n$ and $\infty$ is at least $\delta/3$. It contradicts to Lemma \ref{lem:key} (1).

If $\mathcal{H}$ is of type C, consider the additional critical point $c_n$ of $f_n$ that is not in the immediate basins of $\mathcal{O}_n$. Then $c_n$ converges to $\infty$. In this case, there exists a $k>0$ such that $f_n^k(c_n)$ belongs to the immediate basin of $\mathcal{O}_n$ for all $n$, which stays outside the $\delta/3$ neighborhood of $\infty$.  Again it contradicts to Lemma \ref{lem:key} (1).

If $\mathcal{H}$ is of type D, then $f_n$ has another free (super)attracting cycle $\mathcal{O}_n'\not=\mathcal{O}_n$. Assume $\mathcal{O}_n'$ converges to $\mathcal{O}'$.  By the previous claim, $\mathcal{O}'\subseteq\C$. Since $\hat f$ is a rational map of degree $3$ with $3$ superattracting fixed points, we have $\mathcal{O}'=\mathcal{O}$. It follows that both of the immediate basins of $\mathcal{O}_n$ and $\mathcal{O}_n'$ are contained in the bounded component of $\widehat{\C}\setminus \g_n$. Hence $f_n$ has no additional critical point converging to infinity. It is impossible.

\noindent{{\bf Case 2.(ii)}: $\HHH$ is of type FE1 or FE2.}

In priori, differing from Case 2.(i), the additional critical point of $\hat f$ may be an iterated preimage of $\infty$. So the assumptions of Theorem \ref{main} may fail for the Newton graphs of $\hat f$. Alternatively, we apply Theorem \ref{main} to the Jordan curve $\CCC$ constructed in Section \ref{sec:cubic} in the following argument.

By Lemma \ref{lem:exclude}, the component $\mathcal{H}$ is of inseparable type and the additional critical point $c$ of $\hat f$ is not a pole. Inheriting the notations in Section \ref{sec:cubic}, by Lemma \ref{lem:cubic-curve}, we obtain a Jordan curve $\mathcal{C}$ consisting of some internal rays in $\Omega_1, \Omega_2, \Omega_3, \Omega_1^{(2)}$ and $\Omega_2^{(2)}$ such that the orbits of the landing points of these rays are disjoint with the critical points of $\hat f$ and the bounded component of $\widehat{\C}\setminus\mathcal{C}$ contains $\overline{\Omega}_1^{(1)},\overline{\Omega}_2^{(1)},\overline{\Omega}_3^{(1)},c$ and the poles of $\hat f$.

Set $\mathcal{U}=\{\Omega_1, \Omega_2, \Omega_3, \Omega_1^{(1)}, \Omega_2^{(1)}, \Omega_1^{(2)}, \Omega_2^{(2)}\}$.
Then by Lemma \ref{lem:criterion}, we have $f_n$ converges to $f$ in $\mathcal{U}$ under the dynamically weak Carath\'eodory topology. By Remark \ref{rem:list}.(2), we apply Theorem \ref{main} to $\mathcal{C}$. For all large $n$, we obtain a Jordan curve $\mathcal{C}_n$ consisting of internal rays in $\Omega_{1,n}\cup\Omega_{2,n}\cup\Omega_{3,n}\cup\Omega_{1,n}^{(2)}\cup\Omega_{2,n}^{(2)}$ with the same angles as those in $\Omega_1\cup\Omega_2\cup\Omega_3\cup\Omega_1^{(2)}\cup\Omega_2^{(2)}$. Then the  bounded component of $\widehat{\C}\setminus\mathcal{C}_n$ contains $\overline{\Omega}_{1,n}^{(1)}, \overline{\Omega}_{2,n}^{(1)}, \overline{\Omega}_{3,n}^{(1)}$ and the closures of the two preimages of $\Omega_{4,n}$ disjoint with $\Omega_{4,n}$.
The unbounded component of $\widehat{\C}\setminus\mathcal{C}_n$ contains $\overline{\Omega}_{4,n}$.

Denote by $c_n$ the additional critical point converging to $\infty$. We claim that $c_n$ is in the basin of some root of $f_n$.
To prove this claim, it is sufficient to consider the case that $f_n$ has a free (super)attracting cycle $\mathcal{O}_n$.
Suppose $\mathcal{O}_n$ converges to $\mathcal{O}$. If $\mathcal{O}$ contains $\infty$, the claim follows from the claim in the beginning of Case 2. If $\mathcal{O}\subseteq \C$, by Lemma \ref{lem:limit cycle}, the set $\mathcal{O}$ is the non-repelling cycle of $\hat f$ of period at least $2$. It follows that $\hat f$ is postcritically finite in the basins of the roots and $f^j(c)\not=\infty$ for all $j\geq0$. With the same argument in Case 2.(i), we obtain that the entire immediate basin of $\mathcal{O}_n$ is disjoint with a fixed neighborhood of $\infty$. Hence by Lemma \ref{lem:key} (1), the claim follows.

For the additional critical point $c_n$, there exists a minimal $k>0$ such that $f_n^k(c_n)\in\Omega_{1,n}\cup\Omega_{2,n}\cup\Omega_{3,n}\cup\Omega_{4,n}$. By Lemma \ref{lem:key} (1), for each $0\leq i\leq k-1$ and all large $n$, the Fatou component $U(f_n^i(c_n))$ containing $f_n^i(c_n)$ is not contained in the bounded domain of $\widehat{\C}\setminus\mathcal{C}_n$. Furthermore, none of these Fatou components intersects $\mathcal{C}_n$. Indeed, if $U(f_n^i(c_n))$ intersects $\CCC_n$ for some $0\leq i\leq k-1$, then $U(f_n^i(c_n))$ coincides with either $\Omega_{1,n}^{(2)}$ or $\Omega_{2,n}^{(2)}$. It then follows that $U(f_n^{i+1}(c_n))$ coincides with either $\Omega_{1,n}^{(1)}$ or $\Omega_{2,n}^{(1)}$. Note $\Omega_{1,n}^{(1)}$ and $\Omega_{2,n}^{(1)}$ are both in the bounded component of $\widehat{\C}\setminus\mathcal{C}_n$. It contradicts to Lemma \ref{lem:key} (1). Therefore, for $0\le i\le k-1$, the component $U(f^i_n(c_n))$ is contained in the unbounded component of $\widehat{\C}\setminus\mathcal{C}_n$.

By previous argument, the closure of any non-fixed preimage of $\Omega_{1,n}, \Omega_{2,n}, \Omega_{3,n}$ or $\Omega_{4,n}$ either belongs to the bounded component of $\widehat{\C}\setminus\mathcal{C}_n$, or intersects with $\partial\Omega_{4,n}$ at the pole. Then $\partial U(f_n^{k-1}(c_n))\cap\partial\Omega_{4,n}\not=\emptyset$. Note that $\Omega_{4,n}$ is the unique component of $f_n^{-1}(\Omega_{4,n})$ contained in the unbounded component of $\widehat{\C}\setminus\mathcal{C}_n$. Since each $U(f_n^i(c_n))$ is in the unbounded component of $\widehat{\C}\setminus\mathcal{C}_n$, then $\partial U(f_n^i(c_n))\cap\partial\Omega_{4,n}\not=\emptyset$ for all $0\le i\le k-1$.

Note that all $f_n$ are in the same hyperbolic component, then all quantities defined for $f_n$ and properties satisfied by $f_n$ for $n$ large also hold for $f_0$. We deduce the contradiction by $f_0$. Set $Z_0:=U(c_0)$ and $Z_1:=U(f_0(c_0))$. Suppose $\partial Z_1$ intersects $\partial\Omega_{4,0}$ at the landing point of $I_{4,0}(\theta)$. Since $Z_0$ contains a critical point and is contained in the unbounded component of $\widehat{\C}\setminus\mathcal{C}_0$, the intersection $\partial Z_0\cap \partial\Omega_{4,0}$ contains the landing points of $I_{4,0}(\theta/2)$ and $I_{4,0}((1+\theta)/2)$. We denote by $\g_1$ the arc in $Z_0$ joining these two landing points. Let $\g_2$ be the lift of $\g_1$ based at the landing point of $I_{4,0}(\theta/2^2)$. Since $\g_1$ does not intersect with $\mathcal{C}_0$, the endpoint of $\g_2$ belongs to $\partial\Omega_{4,0}$. Note also that the preimages of $\g_1(1)$ on $\partial\Omega_{4,0}$ are the landing points of the internal rays in $\Omega_{4,0}$ of angles $(1+\theta)/4$ or $(3+\theta)/4$. Since $(1+\theta)/4\in(\theta/2,(1+\theta)/2)$, it follows that the endpoint of $\g_2$ is the landing point of $I_{4,0}((3+\theta)/4)$.

Inductively, for every $k\geq1$, define $\g_{k+1}$ to be the lift of $\g_k$ based at the landing point of $I_{4,0}(\theta/2^{k+1})$. Then the endpoint of $\g_{k+1}$ is the landing point of $I_{4,0}(1-(1-\theta)/2^{k+1})$. Note that, for large $m$, each $\g_m$ is an arc joining two points of $\partial\Omega_{4,0}$ in different components of $\partial\Omega_{4,0}\setminus(I_{4,0}(0)\cup I_{4,0}(1/2))$ near $\infty$. Moreover, the intersection of $\g_m$ and $\overline{\Omega}_{1,0}\cup\overline{\Omega}_{2,0}\cup\overline{\Omega}_{3,0}\cup\overline{\Omega}_{4,0}$ is the endpoints of $\g_m$. It follows that the diameters of $\g_m$ have a positive infinitum as $m\to\infty$. On the other hand, since $f_0$ is uniformly expanding near the Julia set, the diameters of $\g_m$ decrease to $0$ as $m\to\infty$. It is a contradiction.
\end{proof}

\bibliographystyle{siam}
\bibliography{references}
\end{document}